\documentclass[a4paper,12pt]{article}
\usepackage{amssymb,amsmath,amsfonts,amsthm,stmaryrd}

\numberwithin{equation}{section}
\usepackage[left=1 in, right=1 in,top=1 in, bottom=1 in]{geometry}

\providecommand{\abs}[1]{\left\vert#1\right\vert}
\providecommand{\norm}[1]{\left\Vert#1\right\Vert}

\providecommand{\Rn}[1]{\mathbb{R}^{#1}}

\providecommand{\csubset}{\subset\subset}

\providecommand{\jump}[1]{\left\llbracket #1 \right\rrbracket }

\def\nab{\nabla}
\def\dt{\partial_t}

\def\hal{\frac{1}{2}}

\def\ep{\varepsilon}

\def\tep{\tilde{\ep}}
\def\td{\tilde{\delta}}

\def\rest{\hskip 1pt{\hbox to 10.8pt{\hfill\vrule height 7pt width 0.4pt depth 0pt\hbox{\vrule height 0.4pt
width 7.6pt depth 0pt}\hfill}}}

\def\evalu{\hskip 1pt{\hbox to 2pt{\hfill \vrule height -6pt width 0.4pt depth0pt}}}

\DeclareMathOperator{\diverge}{div}
\DeclareMathOperator{\supp}{supp}

\newtheorem{lem}{Lemma}[section]

\newtheorem{prop}[lem]{Proposition}
\newtheorem{thm}[lem]{Theorem}
\newtheorem{remark}[lem]{Remark}

\title{Linear Rayleigh-Taylor instability for viscous, compressible fluids}
\author{Yan Guo\footnote{Supported in part by NSF grant 0603815}\, and Ian Tice\footnote{Supported by an NSF
Postdoctoral Research Fellowship}\\{\small Brown University, Division of Applied Mathematics}\\
{\small 182 George St., Providence, RI 02912}\\
{\small\tt guoy@dam.brown.edu,  tice@dam.brown.edu} }

\begin{document}

\maketitle

\begin{abstract}
We study the equations obtained from linearizing the compressible Navier-Stokes equations around a steady-state profile with a heavier fluid lying above a lighter fluid along a planar interface, i.e. a Rayleigh-Taylor instability.  We consider the equations with or without surface tension, with the viscosity allowed to depend on the density, and in both periodic and non-periodic settings.  In the presence of viscosity there is no natural variational framework for constructing growing  mode solutions to the linearized problem.  We develop a general method of studying a family of modified variational problems in order to produce maximal growing modes.  Using these growing modes, we construct smooth (when restricted to each fluid domain) solutions to the linear equations that grow exponentially in time in Sobolev spaces.  We then prove an estimate for arbitrary solutions to the linearized equations in terms of the fastest possible growth rate for the growing modes.  In the periodic setting, we show that sufficiently small periodicity avoids instability in the presence of surface tension.
\end{abstract}

\section{Formulation of the problem}

\subsection{Formulation in Eulerian coordinates}

We consider two distinct, immiscible, viscous, compressible, barotropic fluids evolving with or without surface tension within the infinite slab $\Omega:= \Rn{2} \times (-m,\ell)\subset \Rn{3}$ with $m,\ell>0$ for time $t\ge 0$.  The fluids are separated from one another by a moving free boundary surface $\Sigma(t)$ that extends to infinity in every horizontal direction; this surface divides  $\Omega$ into two time-dependent, disjoint, open subsets $\Omega_\pm(t)$ so that $\Omega =  \Omega_+(t) \sqcup  \Omega_-(t) \sqcup \Sigma(t)$ and $\Sigma(t) = \bar{\Omega}_+(t) \cap \bar{\Omega}_-(t)$.  The fluid occupying $\Omega_+(t)$ is called the  ``upper fluid,'' and the second fluid, which occupies $\Omega_-(t)$, is called the ``lower fluid.''  The two fluids are described by their density and velocity functions, which are given for each $t\ge 0$ by 
\begin{equation}
  \rho_\pm(\cdot,t) :\Omega_\pm(t) \rightarrow \Rn{+} \text{ and } u_\pm(\cdot,t) :\Omega_\pm(t) \rightarrow \Rn{3}
\end{equation}
respectively.  We shall assume that at a given time $t\ge 0$ the density and velocity functions have well-defined traces onto $\Sigma(t)$.

For $t>0$ and  $x \in \Omega_\pm(t)$ we require that the fluids satisfy the pair of compressible Navier-Stokes equations:
\begin{equation}\label{eulerian_equations}
 \begin{cases}
  \dt \rho_\pm + \diverge(\rho_\pm u_\pm) = 0 \\
  \rho_\pm (\dt u_\pm + u_\pm \cdot \nab u_\pm) + \diverge S_\pm = -g \rho_\pm e_3,  
 \end{cases}
\end{equation}
where the viscous stress tensor is given by
\begin{equation}
 S_\pm = P_\pm(\rho_\pm)  I - \ep_\pm(\rho_\pm) \left(Du_\pm +Du_\pm^T -\frac{2}{3} \diverge{u_\pm} I \right) -\delta_\pm(\rho_\pm)  (\diverge{u_\pm})I.
\end{equation}
In this expression the superscript $T$ means matrix transposition and $I$ is the $3\times 3$ identity matrix.  The coefficients of viscosity are allowed to vary smoothly with the density, i.e. $\ep_\pm,\delta_\pm \in C^\infty((0,\infty))$, but we assume that the shear viscosity satisfies $\ep_\pm>0$ and that the bulk viscosity satisfies $\delta_\pm \ge 0$.  In the equations we have written $g>0$ for the gravitational constant, $e_3 = (0,0,1)$ for the vertical unit vector, and $-ge_3$ for the gravitational force.   We have assumed a general pressure law of the form $P_\pm =  P_\pm(\rho)>0$ with  $P_\pm\in C^\infty((0,\infty))$ and strictly increasing.  We will also assume that $1/P'_{\pm}\in L^\infty_{loc}((0,\infty))$.  Finally, in order to create the Rayleigh-Taylor instability, i.e. construct a steady-state solution with an upper fluid of greater density at $\Sigma(t)$, we will assume that 
\begin{equation}\label{Z_def}
 Z:=\{z\in(0,\infty) \;\vert\; P_-(z) > P_+(z) \text{ and } P_-(z) \in P_+((0,\infty))   \} \neq \varnothing.
\end{equation}
In particular this requires the pressure laws to be distinct, i.e. $P_- \neq P_+$.  For a physical discussion of the Rayleigh-Taylor instability, we refer to to \cite{kull} and the references therein.

For two viscous fluids meeting at a free boundary with surface tension, the standard assumptions are that the velocity is continuous across the interface and the jump in the normal stress is proportional to the mean curvature of the surface multiplied by the normal to the surface (cf.
\cite{we_la}).  This requires us to enforce the jump conditions  
\begin{equation}\label{eulerian_jumps}
 \begin{cases}
   (u_+)\vert_{\Sigma(t)} -  (u_-)\vert_{\Sigma(t)}  = 0 \\
   (S_+ \nu)\vert_{\Sigma(t)} -  (S_- \nu)\vert_{\Sigma(t)}   = \sigma H \nu,
 \end{cases}
\end{equation}
where we have written the normal vector to $\Sigma(t)$ as $\nu$, and $f\vert_{\Sigma(t)}$ for the trace of a quantity $f$ on $\Sigma(t)$.  Here we take $H$ to be twice the mean curvature of the surface $\Sigma(t)$ and the surface tension to be a constant $\sigma\ge0$.  We will also enforce the no-slip condition at the fixed upper and lower boundaries; we implement this via the boundary condition 
\begin{equation}
 u_-(x_1,x_2,-m,t)   = u_+(x_1,x_2,\ell,t)  =0 \text{ for all } (x_1,x_2)\in \Rn{2}, t \ge 0.
\end{equation}

The motion of the free interface is coupled to the evolution equations for the fluids \eqref{eulerian_equations} by requiring that the surface be advected with the fluids.  More precisely, if $V(x,t)\in \Rn{3}$ denotes the normal velocity of the surface at $x\in \Sigma(t)$, then $V(x,t) = (u(x,t)\cdot \nu(x,t))\nu(x,t)$, where $\nu(x,t)$ is the unit normal to $\Sigma(t)$ at $x$ and $u(x,t)$ is the common trace of $u_\pm(\cdot,t)$ onto $\Sigma(t)$.   These traces agree because of the first jump condition in \eqref{eulerian_jumps}, which also implies that there is no possibility of the fluids slipping past each other along $\Sigma(t).$  

To complete the statement of the problem, we must specify initial conditions.  We give the initial interface $\Sigma(0) = \Sigma_0$, which yields the open sets $\Omega_\pm(0)$ on which we specify the initial data for the density and velocity, $\rho_\pm(0):\Omega_\pm(0) \rightarrow \Rn{+}$ and $u_\pm(0):\Omega_\pm(0) \rightarrow \Rn{3}$, respectively.

It is sometimes desirable to add the additional assumption that solutions are periodic in the horizontal directions.  More precisely, we can require that for $L>0$, the domains $\Omega_\pm(t)$ and the free interface $\Sigma(t)$ are horizontally $2\pi L$ periodic in that
\begin{equation}
 \Omega_\pm(t) = \Omega_\pm(t) + 2\pi L k_1 e_1 + 2\pi L k_2 e_2 \text{ and } \Sigma(t) = \Sigma(t) +2\pi L k_1 e_1 + 2\pi L k_2 e_2
\end{equation}
for any $(k_1,k_2)\in \mathbb{Z}^2$.  Then the density and velocity are periodic on $\Omega_\pm(t)$:
\begin{equation}
 \rho_\pm(x + 2\pi L k_1 e_1 + 2\pi L k_2 e_2,t) = \rho_\pm(x,t) \text{ for all } x\in \Omega_\pm(t),
\end{equation}
\begin{equation}
 u_\pm(x + 2\pi L k_1 e_1 + 2\pi L k_2 e_2,t) = u_\pm(x,t) \text{ for all } x\in \Omega_\pm(t).
\end{equation}

\subsection{Reformulation in Lagrangian coordinates}

The movement of the free boundary and the subsequent change of the domains $\Omega_\pm(t)$ in Eulerian coordinates create numerous mathematical difficulties.  We circumvent these by switching to Lagrangian coordinates so that the interface and the domains stay fixed in time.  To this end we define the fixed Lagrangian domains $\Omega_- = \Rn{2} \times (-m,0)$ and $\Omega_+ = \Rn{2} \times (0,\ell)$ in the non-periodic case, and $\Omega_- = (2\pi L \mathbb{T})^2 \times(-m,0)$ and $\Omega_+ = (2\pi L \mathbb{T})^2 \times(0,\ell)$ in the periodic case.  Here we have written $2\pi L \mathbb{T}$ for the $1-$torus of length $2\pi L$.

We assume that there exist mappings 
\begin{equation}
 \eta^0_\pm:\Omega_\pm \rightarrow \Omega_\pm(0)
\end{equation}
that are continuous across $\{x_3=0\}$, invertible in the non-periodic case, and invertible on their image in the periodic case.  We further require that  $\Sigma_0 = \eta^0_\pm(\{x_3=0\})$, $\eta^0_+(\{x_3=\ell\}) = \{x_3=\ell\}$, and $\eta^0_-(\{x_3=-m \}) = \{x_3=-m\}$; the first condition means that $\Sigma_0$ is parameterized by the either of the mappings $\eta^0_\pm$ restricted to $\{x_3=0\}$ (which one is irrelevant since they are continuous across the interface), and the latter two conditions mean that $\eta_\pm^0$ map the fixed upper and lower boundaries into themselves.  

Define the flow maps, $\eta_\pm$, as the solutions to
\begin{equation}
 \begin{cases}
  \dt \eta_\pm(x,t) = u_\pm(\eta_\pm(x,t),t) \\
  \eta(x,0) = \eta_\pm^0(x).
 \end{cases}
\end{equation}
We think of the Eulerian coordinates as $(y,t)$ with $y=\eta(x,t)$, whereas we think of Lagrangian coordinates as the fixed $(x,t)\in \Omega \times \Rn{+}$. In order to switch back and forth from Lagrangian to Eulerian coordinates we assume that $\eta_\pm(\cdot,t)$ are invertible in the non-periodic case and invertible on their images in the periodic case.  In the non-periodic case, this implies that $\Omega_\pm(t) = \eta_\pm(\Omega_\pm,t)$, and since $u_\pm$ and $\eta_\pm^0$ are all continuous across $\{x_3=0\}$, we have $\Sigma(t) = \eta_\pm(\{x_3=0\},t)$.  In other words, the Eulerian domains of upper and lower fluids are the image of $\Omega_\pm$ under the mappings $\eta_\pm$ and  the free interface is the image of $\{x_3=0\}$ under the mapping $\eta_\pm(\cdot,t)$. In the periodic case,
\begin{equation}
 \Omega_\pm(t) = \bigsqcup_{(k_1,k_2)\in \mathbb{Z}^2} \left( \eta_\pm(\Omega_\pm,t) + 2\pi L k_1 e_1 + 2\pi L k_2 e_2 \right), \text{ and }
\end{equation}
\begin{equation}
 \Sigma(t) = \bigsqcup_{(k_1,k_2)\in \mathbb{Z}^2} \left( \eta_\pm(\{x_3=0\} ,t) + 2\pi L k_1 e_1 + 2\pi L k_2 e_2 \right).
\end{equation}

We define the Lagrangian unknowns
\begin{equation}
 \begin{cases}
  v_\pm(x,t) = u_\pm(\eta_\pm(x,t),t) \\
  q_\pm(x,t) = \rho_\pm(\eta_\pm(x,t),t), 
 \end{cases}
\end{equation}
which are defined  for $(x,t) \in \Omega_\pm \times \Rn{+}$.  Since the domains $\Omega_\pm$ are now fixed, we henceforth consolidate notation by writing $\eta, v, q$ to refer to $\eta_\pm, v_\pm, q_\pm$ except when necessary to distinguish the two; when we write an equation for $\eta, v, q$ we assume that the equation holds with the subscripts added on the domains $\Omega_\pm$.  Define the matrix $A$ via $A^T= (D \eta)^{-1}$, where $D$ is the derivative in $x$ coordinates and the superscript $T$ denotes matrix transposition.  Then in Lagrangian coordinates the evolution equations for $v,q,\eta$ are, writing $\partial_j = \partial / \partial x_j$,
\begin{equation}\label{lagrangian_equations}
 \begin{cases}
  \dt \eta_i = v_i \\
  \dt q + q A_{ij}\partial_j v_i =0 \\
  q \dt v_i + A_{jk}\partial_k T_{ij}  = -g q A_{ij}\partial_j \eta_3,
 \end{cases}
\end{equation}
where the viscous stress tensor in Lagrangian coordinates, $T$, is given by
\begin{equation}
 T_{ij} = P(q) I_{ij} - \ep(q) \left(A_{jk} \partial_k v_i + A_{ik} \partial_k v_j  - \frac{2}{3} (A_{lk} \partial_k v_l) I_{ij} \right) -\delta(q)(A_{lk} \partial_k v_l) I_{ij}.
\end{equation}
Here we have written $I_{ij}$ for $i,j$ component of the $3\times 3$ identity matrix $I$ and we have employed the Einstein convention of summing over repeated indices.

To write the jump conditions, for a quantity $f=f_\pm$, we define the interfacial jump as
\begin{equation}
 \jump{f} := f_+ \vert_{\{x_3=0\}} - f_- \vert_{\{x_3=0\}}. 
\end{equation}
The jump conditions across the interface are
\begin{equation}\label{lagrangian_jumps}
 \begin{cases}
  \jump{v} = 0 \\
  \jump{T n}  = \sigma H n
 \end{cases}
\end{equation}
where we have written 
\begin{equation}
n := \left. \frac{\partial_1 \eta \times \partial_2 \eta}{\abs{\partial_1 \eta \times \partial_2 \eta}} \right\vert_{\{x_3=0\}}
\end{equation}
for the unit normal to the surface $\Sigma(t) = \eta(\{x_3=0\},t)$ and $H$ for twice the mean curvature of $\Sigma(t)$.  Since $\Sigma(t)$ is parameterized by $\eta$, we may employ the standard formula  for the mean curvature of a parameterized surface to write
\begin{equation}
 H = \left( \frac{\abs{\partial_1 \eta}^2 \partial_{2}^2 \eta - 2 (\partial_1 \eta \cdot \partial_2 \eta) \partial_1 \partial_2 \eta +  \abs{\partial_2 \eta}^2 \partial_1^2 \eta }{\abs{\partial_1 \eta}^2 \abs{\partial_2 \eta}^2 - \abs{\partial_1 \eta \cdot \partial_2 \eta}^2}   \right) \cdot n.
\end{equation}
Finally, we require the no-slip boundary condition 
\begin{equation}
 v_-(x_1,x_2,-m,t) = v_+(x_1,x_2,\ell,t) = 0.
\end{equation}
Note that the jump and boundary conditions are the same in the periodic and non-periodic cases.

\subsection{Steady-state solution}

We seek a steady-state solution with $v=0, \eta = Id$, $q(x,t) = \rho_0(x_3)$ with the interface given by $\eta(\{x_3=0\})=\{x_3=0\}$ for all $t\ge 0$.  Then $H=0$, $n=e_3$,  and $A=I$ for all $t\ge 0$, and the equations reduce to the ODE
\begin{equation}\label{steady_state}
 \frac{d(P(\rho_0))}{dx_3} = -g\rho_0  
\end{equation}
subject to the jump condition
\begin{equation}
 \jump{ P(\rho_0)} =0.
\end{equation}

To solve this we introduce the enthalpy function defined by 
\begin{equation}
 h_\pm(z) = \int_1^z \frac{P_\pm'(r)}{r}dr.
\end{equation}
The properties of $P_\pm$ guarantee that $h_\pm \in C^\infty((0,\infty))$ are both strictly increasing, and hence invertible on their images. The solution to the ODE is then given by
\begin{equation}
 \rho_0(x) = \begin{cases}
h_-^{-1}(h_-(\rho^-_0) - g x_3 ), & -m < x_3 < 0 \\
h_+^{-1}(h_+(\rho^+_0) - g x_3 ), & 0 < x_3 < \ell.
             \end{cases}
\end{equation}
where $\rho^-_0 >0$ is a free parameter satisfying $P_-(\rho_0^-) \in P_+((0,\infty))$, which allows the jump condition to be satisfied by choosing $\rho_0^+>0$ according to
\begin{equation}
 \rho^+_0 = P_+^{-1}(P_-(\rho^-_0)).
\end{equation}
For $\rho_0$ to be well-defined on $\Omega$, we will henceforth assume that $\ell,m>0$ are chosen so that
\begin{equation}
(h_-(\rho^-_0) + g m) \in h_-((0,\infty)) \text{ and } (h_+(\rho^+_0) - g \ell) \in h_+((0,\infty)).
\end{equation}
Note that $\rho_0$ is bounded above and below by positive constants on $(-m,\ell)$ and that $\rho_0$ is smooth when restricted to $(-m,0)$ or $(0,\ell).$

Since we are interested in Rayleigh-Taylor instability, we want the fluid to be denser above the interface, i.e. $\rho^+_0 > \rho^-_0$.  This requires us to choose $\rho^-_0$ so that
\begin{equation}
P_+^{-1}(P_-(\rho^-_0)) > \rho^-_0 \Leftrightarrow P_-(\rho^-_0) > P_+(\rho^-_0).
\end{equation}
The latter condition is satisfied for any $\rho^-_0 \in Z$, where $Z$ was defined by \eqref{Z_def}; we assume $\rho^-_0$ takes any such value.  Then
\begin{equation}
 \jump{\rho_0} = \rho_0^+ - \rho_0^- >0.
\end{equation}

For the sake of clarity, we include an example of the solution, $\rho_0$, when the pressure laws correspond to polytropic gas laws, i.e. $P_\pm(\rho) = K_\pm \rho^{\gamma_\pm}$ for $K_\pm> 0, \gamma_\pm \ge 1$.  The solution is then given by
\begin{equation}
 \rho_0(x_3) = \begin{cases}
\left((\rho_0^-)^{\gamma_{-}-1} - \frac{g(\gamma_{-}-1)}{K_{-}\gamma_{-}} x_3   \right)^{1/(\gamma_{-}-1)} & x_3 < 0 \\
\left((\rho_0^+)^{\gamma_{+}-1} - \frac{g(\gamma_{+}-1)}{K_{+}\gamma_{+}} x_3   \right)^{1/(\gamma_{+}-1)} &
0<x_3 \le \frac{K_{+}\gamma_{+}}{g(\gamma_{+}-1)}  (\rho_0^+)^{\gamma_{+}-1}     \\
0 & x_3 \ge \frac{K_{+}\gamma_{+}}{g(\gamma_{+}-1)}  (\rho_0^+)^{\gamma_{+}-1} 
\end{cases}
\end{equation}
with modification to solutions $\rho_0(x_3) = \rho_0^{\pm} \exp(-gx_3/K_{\pm})$ when either $\gamma_+$ or $\gamma_-$ is $1$.  The jump condition requires that 
\begin{equation}
 \rho_0^+ = \left(\frac{K_-}{K_+}\right)^{1/\gamma_+} (\rho_0^-)^{\gamma_-/\gamma_+}.
\end{equation}
For a polytropic gas law, the condition that $\rho_0^+>\rho_0^-$ is equivalent to
\begin{equation}
 \left(\frac{K_-}{K_+}\right)^{1/\gamma_+} (\rho_0^-)^{\gamma_-/\gamma_+} > \rho_0^- \Leftrightarrow (\rho_0^-)^{\gamma_- - \gamma_+} > \frac{K_+}{K_-}.
\end{equation}
If $\gamma_+ = \gamma_-$ this requires $K_- > K_+$ and any choice of $\rho_0^->0$.  If $\gamma_+ \neq \gamma_-$ then $K_-, K_+>0$ can be arbitrary, but we must require that $\rho_0^->0$ satisfies
\begin{equation}
\begin{cases}
 \rho_0^- > \left( \frac{K_+}{K_-}\right)^{1/(\gamma_- - \gamma_+)} & \text{if } \gamma_- > \gamma_+\\
 \rho_0^- < \left( \frac{K_-}{K_+}\right)^{1/(\gamma_+ - \gamma_-)} & \text{if } \gamma_+ > \gamma_-.
\end{cases}
\end{equation}
In either case, to avoid the vanishing of $\rho_0$, $\ell$ is chosen so that
\begin{equation}
0< \ell < \frac{K_{+}\gamma_{+}}{g(\gamma_{+}-1)}  (\rho_0^+)^{\gamma_{+}-1},
\end{equation}
but the parameter $m>0$ may be chosen arbitrarily.

\subsection{Linearization around the steady-state}
We now linearize the equations \eqref{lagrangian_equations} around the steady-state solution $v=0$, $\eta = Id$, $q=\rho_0$.  The resulting linearized equations are, writing $\eta, v, q$ for the unknowns,  
\begin{equation}\label{linearized_1}
 \begin{cases}
  \dt \eta  =  v \\
  \dt q + \rho_0 \diverge{ v} =0 
   \end{cases}
\end{equation}
and 
\begin{multline}\label{linearized_2}
\rho_0 \dt v + \nab(P'(\rho_0)q)  + g q e_3 + g \rho_0 \nab \eta_3 \\
= \diverge\left( \ep_0 \left( Dv + Dv^T  -\frac{2}{3} (\diverge{v}) I \right) + \delta_0 (\diverge{v})I \right),
\end{multline}
where $\ep_0 = \ep(\rho_0)$ and $\delta_0 = \delta(\rho_0)$.

The jump conditions linearize to $\jump{v}  =0$ and 
\begin{equation}\label{linearized_jump}
\jump{  P'(\rho_0) q I      - \ep_0 (Dv + Dv^T) - (\delta_0-2\ep_0/3) \diverge{v} I }  e_3  = \sigma \Delta_{x_1,x_2} \eta_3 e_3,
\end{equation}
while the boundary conditions linearize to  $v_-(x_1,x_2,-m,t) = v_+(x_1,x_2,\ell,t)  =0.$
We assume that initial data are provided as $\eta(0) = \eta_0$, $v(0) = v_0$, $q(0) = q_0$ that satisfy the jump and boundary conditions in addition to the assumption that 
$\jump{\eta_0}  =0,$  which implies that $\eta(t)$ is continuous across $\{x_3=0\}$ for all $t\ge 0$.

\subsection{Growing mode ansatz}

We will look for a growing normal mode solution to \eqref{linearized_1}--\eqref{linearized_2} by first assuming an ansatz
\begin{equation}
 v(x,t) = w(x) e^{\lambda t}, q(x,t)= \tilde{q}(x) e^{\lambda t}, \eta(x,t) = \tilde{\eta}(x) e^{\lambda t}
\end{equation}
for some $\lambda > 0$, which is the same in the upper and lower fluids. Plugging the ansatz into \eqref{linearized_1}--\eqref{linearized_2}, we may solve the first and second equations for $\tilde{\eta}$ and $\tilde{q}$ in terms of $v$.  Doing so and eliminating them from the third equation, we arrive at the time-invariant equation
\begin{multline}\label{linear_timeless_0}
 \lambda^2 \rho_0  w  - \nab (P'(\rho_0) \rho_0  \diverge{w} )    
- g\rho_0 \diverge{w}   e_3  + g\rho_0 \nab w_3 \\
= \diverge\left( \lambda \ep_0 \left( Dw + Dw^T  -\frac{2}{3} (\diverge{w}) I \right) + \lambda \delta_0 (\diverge{w})I \right).
\end{multline}
This is coupled to the jump conditions $\jump{w}=0$ and
\begin{equation}
\jump{ (\lambda \delta_0 - 2\lambda \ep_0/3 +  P'(\rho_0) \rho_0)\diverge{w} I + \lambda\ep_0 (Dw + Dw^T) } e_3   = - \sigma  \Delta_{x_1,x_2} w_3 e_3,
\end{equation}
and the boundary conditions $w_-(x_1,x_2,-m) = w_+(x_1,x_2,\ell)  =0.$  Notice that the first jump condition implies that the assumptions on $\eta(0) = \tilde{\eta}(0) = w(0)/\lambda$ mentioned in the last section are satisfied.

Since the coefficients of the linear problem \eqref{linear_timeless_0} only depend on the $x_3$ variable, we are free to make the further structural assumption that the $x_1,x_2$ dependence of $w$ is given as a Fourier mode $e^{i x' \cdot \xi}$, where $x' \cdot \xi = x_1 \xi_1 + x_2 \xi_2$ for $\xi \in \Rn{2}$ in the non-periodic case and $\xi \in L^{-1} \mathbb{Z} \times L^{-1} \mathbb{Z}$ in the $2\pi L$ periodic case.  Together with the growing mode ansatz, this constitutes a ``normal mode'' ansatz, which is standard in fluid stability analysis \cite{chandra}.  We define the new unknowns $\varphi,\theta,\psi:(-m,\ell) \rightarrow \Rn{}$ according to
\begin{equation}
 w_1(x)  =  -i \varphi(x_3) e^{i x' \cdot \xi},  w_2(x) =    - i \theta(x_3) e^{i x' \cdot \xi}, \text{ and } w_3(x) = \psi(x_3)e^{i x' \cdot \xi}.
\end{equation}

The utility of the new unknowns is seen in the pair of equations
\begin{equation}
 \diverge{w} = (\xi_1 \varphi + \xi_2 \theta + w_3')e^{i x'\cdot \xi}
\end{equation}
and
\begin{equation}
 Dw + Dw^T = \begin{pmatrix}
       2\xi_1 \varphi & \xi_1 \theta + \xi_2 \varphi & i(\xi_1 \psi - \varphi' ) \\
       \xi_1 \theta + \xi_2 \varphi & 2\xi_2 \theta & i(\xi_2 \psi - \theta')  \\
       i(\xi_1 \psi - \varphi' ) & i(\xi_2 \psi - \theta') & 2\psi'
      \end{pmatrix} e^{i x'\cdot \xi}
\end{equation}
For each fixed $\xi$, and for the new unknowns $\varphi(x_3), \theta(x_3), \psi(x_3)$, and $\lambda$ we arrive at the following system of ODEs (here $' = d/dx_3$).
\begin{multline}\label{w_1_equation}
-\left(\lambda \ep_0 \varphi'\right)' + \left[ \lambda^2 \rho_0  + \lambda \ep_0 \abs{\xi}^2 + \xi_1^2 \left( \lambda \delta_0 + \lambda \ep_0/3+ P'(\rho_0) \rho_0 \right)   \right] \varphi \\ = - \xi_1 \left[ \left(\lambda \delta_0 + \lambda \ep_0/3+ P'(\rho_0)\rho_0\right)\psi' +(\lambda \ep_0'  - g\rho_0) \psi\right] 
 -\xi_1 \xi_2  \left[ \lambda \delta_0 + \lambda \ep_0/3 + P'(\rho_0) \rho_0  \right] \theta
\end{multline}
\begin{multline}\label{w_2_equation}
-(\lambda \ep_0 \theta')' + \left[ \lambda^2 \rho_0  + \lambda \ep_0 \abs{\xi}^2 + \xi_2^2 \left( \lambda \delta_0 + \lambda \ep_0/3 +  P'(\rho_0) \rho_0 \right)   \right] \theta \\ = - \xi_2\left[ \left(\lambda \delta_0 + \lambda \ep_0/3+ P'(\rho_0)\rho_0\right)\psi' +(\lambda \ep_0' - g\rho_0) \psi\right] 
 -\xi_1 \xi_2  \left[ \lambda \delta_0 + \lambda \ep_0/3 +  P'(\rho_0) \rho_0  \right] \varphi
\end{multline}
\begin{multline}\label{w_3_equation}
 -\left[\left(  4\lambda \ep_0/3 + \lambda \delta_0 +  P'(\rho_0) \rho_0   \right) \psi'\right]' 
+ \left( \lambda^2 \rho_0 + \lambda \ep_0 \abs{\xi}^2  \right) \psi \\
=  \left[  \left( \lambda \delta_0 + \lambda \ep_0/3 +   P'(\rho_0) \rho_0 \right) \left( \xi_1 \varphi + \xi_2 \theta \right)  \right]' + (g \rho_0 - \lambda \ep_0')(\xi_1 \varphi + \xi_2 \theta )
\end{multline}

The first jump condition yields jump conditions for the new unknowns:
\begin{equation}
 \jump{\varphi} =\jump{\theta} =\jump{\psi} = 0.
\end{equation}
The second jump condition becomes
\begin{equation}
\jump{ \left(\lambda \delta_0 - 2\lambda \ep_0 /3 +   P'(\rho_0) \rho_0 \right)(\xi_1 \varphi +  \xi_2 \theta + \psi') e_3 + \lambda \ep_0 \begin{pmatrix}
   i (\xi_1 \psi - \varphi') \\
  i (\xi_2 \psi-\theta') \\
  2 \psi'
\end{pmatrix}
 }  
    = \sigma \abs{\xi}^2 \psi e_3,
\end{equation}
 which implies that 
\begin{equation}
 \jump{\lambda \ep_0 (\varphi'-\xi_1  \psi)} = \jump{\lambda \ep_0 (\theta'-\xi_2 \psi)} =0 
\end{equation}
and that 
\begin{equation}
\jump{(\lambda \delta_0+ \lambda \ep_0/3 + P'(\rho_0)\rho_0) (\psi' + \xi_1 \varphi + \xi_2 \theta )}
+  \jump{\lambda \ep_0 \left(\psi' - \xi_1 \varphi - \xi_2 \theta    \right)  }
  =\sigma \abs{\xi}^2 \psi.
\end{equation}
The boundary conditions 
\begin{equation}
 \varphi(-m) =  \varphi(\ell) = \theta(-m) =  \theta(\ell)  =\psi(-m) =  \psi(\ell)=0
\end{equation}
must also hold.

We can reduce the complexity of the problem by removing the component $\theta$.  To do this, note that if $\varphi,\theta,\psi$ solve the equations \eqref{w_1_equation}--\eqref{w_3_equation} for $\xi\in \Rn{2}$ and $\lambda$, then for any rotation operator $R\in SO(2)$,  $(\tilde{\varphi},\tilde{\theta}) := R (\varphi,\theta)$ solve the same equations for $\tilde{\xi} := R\xi$ with $\psi, \lambda$ unchanged.  So, by choosing an appropriate rotation, we may assume without loss of generality that $\xi_2=0$ and $\xi_1= \abs{\xi} \ge 0$.  In this setting $\theta$ solves   
\begin{equation}
\begin{cases}
 -(\lambda \ep_0 \theta')' + (\lambda^2 \rho_0 + \lambda \ep_0 \abs{\xi}^2) \theta =0 \\
 \theta(-m) = \theta(\ell) = 0 \\
 \jump{\theta} = \jump{\lambda \ep_0 \theta'} =0.
\end{cases}
\end{equation}
Multiplying this equation by $\theta$, integrating over $(-m,\ell)$, integrating by parts, and using the jump conditions then yields
\begin{equation}
 \int_{-m}^\ell \lambda \ep_0 \abs{\theta'}^2 + (\lambda^2 \rho_0 + \lambda \ep_0 \abs{\xi}^2) \theta^2 =0,
\end{equation}
which implies that $\theta=0$ since we assume $\lambda >0$.  This reduces to the pair of equations for $\varphi,\psi$
\begin{multline}\label{coupled_1}
-\lambda^2 \rho_0 \varphi = -(\lambda \ep_0 \varphi')' + \abs{\xi}^2  \left(4 \lambda \ep_0/3 +   \lambda \delta_0 +  P'(\rho_0) \rho_0 \right) \varphi \\
+ \abs{\xi} \left[ \left(\lambda \delta_0 +\lambda \ep_0/3 + P'(\rho_0)\rho_0\right)\psi' +(\lambda \ep_0' - g\rho_0) \psi \right]  
\end{multline}
\begin{multline}\label{coupled_2}
 -\lambda^2 \rho_0 \psi =  -\left[ \left(  4\lambda \ep_0/3 + \lambda \delta_0 +  P'(\rho_0) \rho_0   \right) \psi' \right]' +   \lambda \ep_0 \abs{\xi}^2  \psi \\
- \abs{\xi} \left[ \left(\left( \lambda \delta_0 +\lambda \ep_0/3 +  P'(\rho_0) \rho_0    \right) \varphi \right)'  + (g \rho_0-\lambda \ep_0' )\varphi \right]  
\end{multline}
along with the jump conditions
\begin{equation}\label{coupled_conds_1}
\jump{ \varphi } = \jump{\psi }= \jump{\lambda \ep_0 (\varphi'-\abs{\xi} \psi)} = 0, 
\end{equation}
\begin{equation}
\jump{(\lambda \delta_0+ \lambda \ep_0/3 + P'(\rho_0)\rho_0) (\psi' + \abs{\xi} \varphi  )}
+  \jump{\lambda \ep_0 \left(\psi' - \abs{\xi} \varphi     \right)  }
  =\sigma \abs{\xi}^2 \psi.
\end{equation}
and the boundary conditions
\begin{equation}\label{coupled_conds_2}
 \varphi(-m) =  \varphi(\ell) =  \psi(-m) =  \psi(\ell)=0.
\end{equation}

\section{Main results and discussion}

In the absence of viscosity ($\ep = \delta =0$ with modified jump and boundary conditions) and for a fixed spatial frequency $\xi\neq 0$, the equations \eqref{coupled_1}--\eqref{coupled_2} can be viewed as an eigenvalue problem with eigenvalue $-\lambda^2$.   Such a problem has a natural variational structure that allows for construction of solutions via the direct methods and for a variational characterization of the eigenvalue via
\begin{equation}
 -\lambda^2 = \inf \frac{E(\varphi,\psi)}{J(\varphi,\psi)},
\end{equation}
where
\begin{equation}
 E(\varphi,\psi)= \hal \int_{-m}^\ell  P'(\rho_0)\rho_0 (\psi' + \abs{\xi} \varphi)^2    
- 2  g\rho_0  \abs{\xi}  \psi \varphi  
\end{equation}
and 
\begin{equation}
 J(\varphi,\psi) = \hal \int_{-m}^\ell \rho_0  (\varphi^2 + \psi^2).
\end{equation}
This variational structure was essential to our analysis in \cite{guo_tice}, where we showed that $\lambda \rightarrow \infty$ as $\abs{\xi} \rightarrow \infty$, which led to ill-posedness results for both the inviscid linearized problem and the inviscid non-linear problem (equations  \eqref{lagrangian_equations} with $\ep=\delta=0$).

Unfortunately, when viscosity is present the natural variational structure breaks down since $\lambda$ appears quadratically as a multiplier of $\rho_0$ and  linearly as a multiplier of $\ep_0 $ and $\delta_0$ in \eqref{coupled_1}--\eqref{coupled_2}.  This  presents no obstacle to a stability analysis once a solution is  known \cite{chandra} since the equations imply a quadratic relationship between $\lambda$ and various integrals of the solution, which can be solved for $\lambda$ to determine the sign of $\Re{\lambda}$.  On the other hand, the appearance of $\lambda$ both quadratically and linearly  eliminates the capacity to use  constrained minimization techniques to produce solutions to the equations.

In order to circumvent this problem and restore the ability to use variational methods, we artificially remove the linear dependence on $\lambda$.  To this end, we define the modified viscosities $\tep = s \ep_0$ and $\td = s \delta_0$, where   $s >0$ is an arbitrary parameter.  We then introduce a family ($s>0$) of modified problems given by  
\begin{multline}\label{s_coupled_1}
-\lambda^2 \rho_0 \varphi = -(\tep \varphi')' + \abs{\xi}^2  \left(4 \tep/3 +   \td +  P'(\rho_0) \rho_0 \right) \varphi \\
+ \abs{\xi} \left[ \left(\td +\tep/3 + P'(\rho_0)\rho_0\right)\psi' +(\tep' - g\rho_0) \psi \right] 
\end{multline}
\begin{multline}\label{s_coupled_2}
 -\lambda^2 \rho_0 \psi =  -\left[ \left(  4 \tep/3 + \td +  P'(\rho_0) \rho_0   \right) \psi' \right]' +   \tep \abs{\xi}^2  \psi \\
- \abs{\xi} \left[ \left(\left( \td +\lambda \ep_0/3 +  P'(\rho_0) \rho_0    \right) \varphi \right)'  + (g \rho_0- \tep' )\varphi \right]  
\end{multline}
along with the jump conditions
\begin{equation}\label{s_coupled_conds_1}
\jump{ \varphi } = \jump{\psi }= \jump{\tep (\varphi'-\abs{\xi} \psi)} = 0, 
\end{equation}
\begin{equation}
\jump{(\td + \tep/3 + P'(\rho_0)\rho_0) (\psi' + \abs{\xi} \varphi  )}
+  \jump{\tep \left(\psi' - \abs{\xi} \varphi     \right)  }
  =\sigma \abs{\xi}^2 \psi.
\end{equation}
and the boundary conditions
\begin{equation}\label{s_coupled_conds_2}
 \varphi(-m) =  \varphi(\ell) =  \psi(-m) =  \psi(\ell)=0.
\end{equation}
A solution to the modified problem with $\lambda = s$ corresponds to a solution to the original problem.

Modifying the problem in this way restores the variational structure and allows us to apply a constrained minimization to the viscous analog of the energy $E$ defined above (see \eqref{E_def}) to find a solution to \eqref{s_coupled_1}--\eqref{s_coupled_2} with $\lambda = \lambda(\abs{\xi},s)>0$ when $s>0$ is sufficiently small and precisely when  
\begin{equation}
 0 < \abs{\xi} \le \abs{\xi}_c : = \sqrt{\frac{g\jump{\rho_0}}{\sigma}}.
\end{equation}
We then further exploit the variational structure to show that $\lambda$ is a continuous function and is strictly increasing in $s$.  Using this, we show in Theorem \ref{inversion} that the parameter $s$ can be uniquely chosen so that 
\begin{equation}
s = \lambda(\abs{\xi},s), 
\end{equation}
which implies that we have found a solution to the original problem \eqref{coupled_1}--\eqref{coupled_2}.  This choice of $s$ allows us to think of $\lambda = \lambda(\abs{\xi})$, and gives rise to a solution to the system of equations  \eqref{w_1_equation}--\eqref{w_3_equation} as well.

\begin{thm}[Proved in Section \ref{mod_section}]\label{w_soln_2}
For $\xi\in\Rn{2}$ so that $0 < \abs{\xi}^2 < g\jump{\rho_0}/\sigma$ there exists a solution $\varphi = \varphi(\xi,x_3)$, $\theta = \theta(\xi,x_3)$, $\psi = \psi(\xi,x_3)$, and $\lambda = \lambda(\abs{\xi})>0$  to \eqref{w_1_equation}--\eqref{w_3_equation} satisfying the appropriate jump and boundary conditions so that  $\psi(\xi,0)\neq 0$.
The solutions are smooth when restricted to $(-m,0)$ or $(0,\ell)$, and they are equivariant in $\xi$ in the sense that if $R\in SO(2)$ is a rotation operator, then 
\begin{equation}\label{w_s_0}
\begin{pmatrix}
\varphi(R \xi,x_3) \\ \theta(R \xi,x_3) \\ \psi(R \xi,x_3) 
\end{pmatrix}
= 
\begin{pmatrix}
R_{11} & R_{12} & 0 \\ 
R_{21} & R_{22} & 0 \\ 
0      & 0      & 1
\end{pmatrix}
\begin{pmatrix}
\varphi(\xi,x_3) \\ \theta(\xi,x_3) \\ \psi(\xi,x_3) 
\end{pmatrix}.
\end{equation}
\end{thm}

Without surface tension ($\sigma=0$) it is possible to construct a solution to \eqref{s_coupled_1}--\eqref{s_coupled_2} with $\lambda>0$ for any $\xi \neq 0$, but with surface tension ($\sigma>0$) there is a critical frequency $\abs{\xi}_c = \sqrt{g\jump{\rho_0}/\sigma}$ for which no solution with $\lambda >0$ is available if $\abs{\xi}\ge \abs{\xi}_c$.  In the non-periodic case, we capture a continuum $\abs{\xi} \in (0,\abs{\xi}_c)$ of growing mode solutions, but in the $2\pi L$ periodic case we only find finitely many.  Indeed, if 
\begin{equation}\label{L_large}
 \sqrt{\frac{\sigma}{g\jump{\rho_0}}} < L,
\end{equation}
then a positive but finite number of spatial frequencies  $\xi \in (L^{-1}\mathbb{Z})^2$ satisfy $\abs{\xi} < \abs{\xi}_c$.  On the other hand, if 
\begin{equation}\label{L_small}
 L \le \sqrt{\frac{\sigma}{g\jump{\rho_0}}},
\end{equation}
then our method fails to construct any growing mode solutions at all.  

It is important to know the behavior of $\lambda(\abs{\xi})$ as $\abs{\xi}$ varies within $0< \abs{\xi} < \abs{\xi}_c$.  We show in Proposition \ref{lambda_cont} that $\lambda(\abs{\xi})$ is continuous and satisfies 
\begin{equation}
 \lim_{\abs{\xi} \rightarrow 0} \lambda(\abs{\xi}) = \lim_{\abs{\xi} \rightarrow \abs{\xi}_c} \lambda(\abs{\xi}) =0.
\end{equation}
In the non-periodic case, this implies that there is a largest growth rate
\begin{equation}\label{max_def}
 0< \Lambda := \max_{0 \le \abs{\xi} \le \abs{\xi}_c} \lambda(\abs{\xi}),
\end{equation}
and in the periodic case for $L$ satisfying \eqref{L_small} the largest rate is
\begin{equation}\label{max_def_per}
0< \Lambda_L:= \sup \{ \lambda(\abs{\xi}) \;\vert\; \xi \in (L^{-1}\mathbb{Z})^2 \text{ and } \abs{\xi} \in (0,\abs{\xi}_c)    \}.
\end{equation}
Note that in general  $\Lambda_L < \Lambda$.  In either case, the largest growth rate is achieved for some particular choice of $\xi$.    

The stabilizing effects of viscosity and surface tension are evident in these results.  As we showed in \cite{guo_tice}, without viscosity or surface tension, $\lambda(\abs{\xi})\rightarrow \infty$ as $\abs{\xi} \rightarrow \infty$.  With viscosity but no surface tension, all spatial frequencies remain unstable, but the growth rate $\lambda(\abs{\xi})$ is bounded and decays to $0$ as $\abs{\xi} \rightarrow \infty$.   With viscosity and surface tension, only a critical interval of spatial frequencies are unstable, and $\lambda(\abs{\xi})$ remains bounded.  Finally, with viscosity and surface tension and the periodicity $L$ satisfying \eqref{L_small} there do not exist any growing modes.

In the periodic case when $L$ satisfies \eqref{L_large}, the solutions to \eqref{w_1_equation}--\eqref{w_3_equation} constructed in Theorem \ref{w_soln_2} immediately give rise to growing mode solutions to \eqref{linearized_1}--\eqref{linearized_2}.

\begin{thm}[Proved in Section \ref{growing_section}]\label{growing_mode_soln_periodic}
Suppose that $L$ satisfies \eqref{L_large} and let $\xi_1,\xi_2 \in (L^{-1}\mathbb{Z})^2$ be lattice points  such that $\xi_1 = -\xi_2$ and   $\lambda(\abs{\xi_i}) = \Lambda_L,$
where $\Lambda_L$ is defined by \eqref{max_def_per}. Define 
\begin{equation}
 \hat{w}(\xi,x_3) =  -i \varphi(\xi,x_3) e_1 - i \theta(\xi,x_3) e_2 + \psi(\xi,x_3) e_3,
\end{equation}
where $\varphi,\theta,\psi$ are the solutions provided by Theorem \ref{w_soln_2}.   Writing $x' = x_1 e_1 + x_2 e_2$, we define 
\begin{equation}
 \eta(x,t) =  e^{\Lambda_L t}  \sum_{j=1}^2 \hat{w}(\xi_j,x_3)  e^{i x'\cdot \xi_j},
\end{equation}
\begin{equation}
 v(x,t) = \Lambda_L e^{\Lambda_L t} \sum_{j=1}^2  \hat{w}(\xi_j,x_3)  e^{i x'\cdot \xi_j},
\end{equation}
and
\begin{equation}
 q(x,t) 
= - e^{\Lambda_L t} \rho_0(x_3) \sum_{j=1}^2  (  e_1 \cdot \xi_j \varphi(\xi_j,x_3) +  e_2 \cdot \xi_j \theta(\xi_j,x_3) + \partial_3 \psi(\xi_j,x_3))   e^{i x'\cdot \xi_j}.
\end{equation}
Then $\eta,v,q$ are real solutions to  \eqref{linearized_1}--\eqref{linearized_2} and the corresponding jump and boundary conditions. For every $t \ge 0$ we have  $\eta(t),v(t),q(t) \in H^k(\Omega)$ and  
\begin{equation}\label{gmsp_0}
\begin{cases}
 \norm{\eta(t)}_{H^k} = e^{t \Lambda_L } \norm{\eta(0)}_{H^k}     \\
 \norm{v(t)}_{H^k}  =  e^{t \Lambda_L} \norm{v(0)}_{H^k}  \\
 \norm{q(t)}_{H^k}  = e^{t \Lambda_L} \norm{q(0)}_{H^k} 
\end{cases}
\end{equation}
\end{thm}

\begin{remark}
 In this theorem, the space $H^k(\Omega)$ is not the usual Sobolev space of order $k$, but what we call the piecewise Sobolev space of order $k$.  See \eqref{sob_def} for the precise definition.
\end{remark}

In the non-periodic case, although $\Lambda = \lambda(\abs{\xi})$ for some $\abs{\xi}\in(0,\abs{\xi}_c)$, no $L^2(\Omega)$ solution to \eqref{linearized_1}--\eqref{linearized_2} may be constructed from a solution to \eqref{w_1_equation}--\eqref{w_3_equation} as in the periodic case since $e^{i x'\cdot \xi} \notin L^2(\Omega)$.  We get around this problem by utilizing a Fourier synthesis of such solutions.  The tradeoff for getting $L^2(\Omega)$ solutions is that the growth rate is not exactly $e^{\Lambda t}$.  Nevertheless,  it is possible to construct solutions that grow arbitrarily close to this rate.

\begin{thm}[Proved in Section \ref{growing_section}]\label{growing_mode_soln}
Let $f\in C_c^\infty((0,\abs{\xi}_c))$ be a real-valued function.  For $\xi \in \Rn{2}$ with $\abs{\xi} \in(0,\abs{\xi}_c)$ define
\begin{equation}
 \hat{w}(\xi,x_3) =  -i \varphi(\xi,x_3) e_1 - i \theta(\xi,x_3) e_2 + \psi(\xi,x_3) e_3,
\end{equation}
where $\varphi,\theta,\psi$ are the solutions provided by Theorem \ref{w_soln_2}.  Writing $x' = x_1 e_1 + x_2 e_2$, we define
\begin{equation}\label{g_m_s_1}
 \eta(x,t) = \frac{1}{4\pi^2} \int_{\Rn{2}} f(\abs{\xi}) \hat{w}(\xi,x_3) e^{\lambda(\abs{\xi}) t} e^{i x'\cdot \xi} d\xi,
\end{equation}
\begin{equation}
 v(x,t) = \frac{1}{4\pi^2}  \int_{\Rn{2}}\lambda(\abs{\xi}) f(\abs{\xi}) \hat{w}(\xi,x_3) e^{\lambda(\abs{\xi}) t} e^{i x'\cdot \xi} d\xi ,
\end{equation}
and
\begin{equation}\label{g_m_s_2}
 q(x,t) 
= -\frac{\rho_0(x_3)}{4\pi^2}  \int_{\Rn{2}} f(\abs{\xi}) (\xi_1 \varphi(\xi,x_3) +  \xi_2 \theta(\xi,x_3) + \partial_{x_3} \psi(\xi,x_3)) e^{\lambda(\abs{\xi}) t} e^{i x'\cdot \xi} d\xi.
\end{equation}
Then $\eta,v,q$ are real-valued solutions to the linearized equations \eqref{linearized_1}--\eqref{linearized_2} along with the corresponding jump and boundary conditions.  The solutions are equivariant in the sense that if $R\in SO(3)$ is a rotation that keeps the vector $e_3$ fixed, then
\begin{equation}\label{g_m_s_00}
\eta(Rx,t) = R \eta(x,t), v(Rx,t) = R v(x,t), \text{and } q(Rx,t) = q(x,t).
\end{equation}
For every $k \in \mathbb{N}$ we have the estimate
\begin{equation}\label{g_m_s_0} 
\norm{\eta(0)}_{H^k} + \norm{v(0)}_{H^k} + \norm{q(0)}_{H^k} \le \bar{C}_k \left( \int_{\Rn{2}} (1+\abs{\xi}^2)^{k+1} \abs{f(\xi)}^2 d\xi \right)^{1/2} < \infty
\end{equation}
for a constant $\bar{C}_k>0$ depending on the parameters $\rho^\pm_0, P_{\pm}, g,\sigma, m, \ell$;  moreover, for every $t > 0$ we have  $\eta(t),v(t),q(t) \in H^k$ and  
\begin{equation}\label{g_m_s_3}
\begin{cases}
e^{t \lambda_0(f)} \norm{\eta(0)}_{H^k} \le   \norm{\eta(t)}_{H^k} \le e^{t \Lambda } \norm{\eta(0)}_{H^k}     \\
e^{t \lambda_0(f)} \norm{v(0)}_{H^k} \le \norm{v(t)}_{H^k}  \le  e^{t \Lambda} \norm{v(0)}_{H^k}  \\
e^{t \lambda_0(f)} \norm{q(0)}_{H^k} \le \norm{q(t)}_{H^k}  \le e^{t \Lambda} \norm{q(0)}_{H^k} 
\end{cases}
\end{equation}
where
\begin{equation}
 \lambda_0(f) = \inf_{\abs{\xi} \in \supp(f)} \lambda(\abs{\xi}) >0
\end{equation}
and $\Lambda$ is given by \eqref{max_def}.
\end{thm}

The vertical component of the initial linearized flow map at the interface between the two fluids is given in the periodic case by 
\begin{equation}
 \eta_3(x_1,x_2,0,0) =   2 \psi( \xi_1 ,0)  \cos( x'\cdot \xi_1),
\end{equation}
and in  the non-periodic case by
\begin{equation}
 \eta_3(x_1,x_2,0,0) = \frac{1}{4\pi^2}  \int_{\Rn{2}} f(\abs{\xi})  \psi( \xi ,0) \cos( x'\cdot \xi) d\xi.
\end{equation}
Since $\psi(\xi,0)\neq 0$ for any choice of $\xi$, a nonzero $f$ in general gives rise to a nonzero $\eta_3(x_1,x_2,0,0)$ in the non-periodic case, and $\eta_3(x_1,x_2,0,0)$ cannot vanish identically in the periodic case.  From this we see that vertical displacement is essential to our unstable solutions.

It is conceivable that the solutions we construct via the modified viscosity trick somehow fail to achieve the fastest growing modes, and so it is not obvious that the growing solutions constructed in Theorems \ref{growing_mode_soln_periodic} and \ref{growing_mode_soln} grow in time at the fastest rate possible.  Nevertheless, this result is true.  In the non-periodic case and in the periodic case when $L$ satisfies \eqref{L_large}, we can estimate the growth in time of arbitrary solutions to \eqref{linearized_1}--\eqref{linearized_2} in terms of $\Lambda$ and $\Lambda_L$.  The technique we employ was inspired by a similar result, proved in \cite{guo_hw}, for the inviscid, incompressible regime with smooth density profile.

To state the result, we first define the weighted $L^2$ norm and the viscosity seminorm by
\begin{equation}\label{norm_def}
 \norm{v}^2_1 = \int_\Omega \rho_0 \abs{v}^2 \; \text{ and } \;
 \norm{v}^2_2 = \int_\Omega \frac{\ep_0}{2} \abs{Dv+ Dv^T - \frac{2}{3}(\diverge{v})I}^2 + \delta_0 \abs{\diverge{v}}^2
\end{equation}
and for $i=1,2$ we write $\langle \cdot, \cdot \rangle_{i}$  for the inner-product giving rise to each.

\begin{thm}[Proved in Section \ref{linear_growth_section}]\label{lin_growth_bound}
Let $v,\eta,q$ be a solution to \eqref{linearized_1}--\eqref{linearized_2} along with the corresponding jump and boundary conditions.  Then in the non-periodic case
\begin{multline}
\norm{ v(t)}_1^2 + \norm{v(t)}_2^2  +\norm{\dt v(t)}_1^2 \\
 \le  C e^{2\Lambda t}\left(\norm{\dt v(0)}_1^2 + \norm{ v(0)}_1^2 + \norm{v(0)}_2^2  +\sigma  \int_{\Rn{2}} \abs{\nab_{x_1,x_2} v_3(0)}^2 \right)
\end{multline}
for a constant $0<C = C(\rho_0^{\pm},P_{\pm},\Lambda,\ep,\delta,\sigma,g,m,\ell)$.  In the periodic case with $L$ satisfying \eqref{L_large}, the same inequality holds with $\Lambda$ replaced with $\Lambda_L$ and the integral over $\Rn{2}$ replaced with an integral over $(2\pi L \mathbb{T})^2$.
\end{thm}

In the periodic case, when there is surface tension and  $L$ satisfies \eqref{L_small}, our method fails to construct any growing mode solutions.  A priori this does not rule out exponential-in-time growth of arbitrary solutions to the linearized equations, but it turns out that  exponential growth is impossible, and a sort of stability estimate is available.

\begin{thm}[Proved in Section \ref{linear_growth_section}]\label{periodic_stability}
In the periodic case let $L$ satisfy \eqref{L_small}.  For $j\ge 1$ define the constants $K_j\ge 0$ in terms of the initial data via
\begin{multline}
 K_j = \int_\Omega \rho_0 \frac{\abs{\dt^j v(0)}^2}{2} + \int_\Omega \frac{P'(\rho_0) \rho_0}{2}\abs{ \diverge{\dt^{j-1} v(0)} - \frac{g}{P'(\rho_0)} \dt^{j-1} v_3(0) }^2 \\
+ \int_{(2\pi L \mathbb{T})^2} \frac{\sigma}{2}\abs{\nab_{x_1,x_2} \dt^{j-1} v_3(0)}^2.
\end{multline}
Then solutions to \eqref{linearized_1}--\eqref{linearized_2} satisfy
\begin{equation}
 \norm{\eta(t)}_1  + \norm{\eta(t)}_2 \le \norm{\eta(0)}_1 + \norm{\eta(0)}_2 + t\left(\norm{v(0)}_1 + \norm{v(0)}_2 \right) + 2t^{3/2} \sqrt{K_1},
\end{equation}
\begin{equation}
 \norm{v(t)}_1  + \norm{v(t)}_2 \le \norm{v(0)}_1 + \norm{v(0)}_2 + 3\sqrt{t}\sqrt{K_1},
\end{equation}
and  for $j\ge 1$
\begin{equation}\label{p_s_0}
 \sup_{t\ge 0} \hal \norm{\dt^j v(t) }_1^2 + \int_0^\infty \norm{\dt^j v(t)}_2^2 dt \le 2 K_j
\end{equation}
and
\begin{equation}\label{p_s_00}
 \sup_{t\ge 0} \norm{\dt^j v(t)}^2_2 \le \norm{\dt^j v(0)}^2_2 + 2 \sqrt{K_j}\sqrt{K_{j+1}}.
\end{equation}
\end{thm}

Our method of studying a family of modified variational problems in order to produce  growing  solutions to linearized problems where viscosity has destroyed the proper variational structure is quite general and robust.  The method may be used to construct growing solutions to the compressible Navier-Stokes-Poisson equations with viscosity (cf. \cite{lin}), and we expect it to be useful for many other viscous, compressible hydrodynamic stability problems.  The linear instability analysis of this paper comprises the first step in an analysis of the non-linear instability of the full equations \eqref{lagrangian_equations}, which will be completed in \cite{guo_tice_2}.  A non-linear instability analysis of the compressible Navier-Stokes-Poisson equations based on linear growing solutions constructed using our method will be completed in \cite{jang_tice} for the case of constant viscosity and in  \cite{g_l_t} for the case of density-dependent viscosity.

The plan of the paper is as follows.  In Section 3 we study the family of modified variational problems in order to produce growing solutions to \eqref{linearized_1}--\eqref{linearized_2}.  In Section 4 we prove the growth estimates for arbitrary solutions to the linearized problem.

\section{A family of modified variational problems}

\subsection{Solutions to \eqref{w_1_equation}--\eqref{w_2_equation} via constrained minimization}\label{mod_section}

In this section we will produce a solution to \eqref{w_1_equation}--\eqref{w_2_equation} with fixed $\abs{\xi}>0$  by first utilizing variational methods to construct solutions to the modified problem \eqref{s_coupled_1}--\eqref{s_coupled_2}. In order to understand $\lambda$ in a variational framework we consider the two energies
\begin{multline}\label{E_def}
E(\varphi,\psi) = \frac{\sigma \abs{\xi}^2}{2}(\psi(0))^2 
+  \hal \int_{-m}^\ell (\td + P'(\rho_0) \rho_0) (\psi' + \abs{\xi} \varphi)^2    
- 2  g\rho_0  \abs{\xi} \varphi \psi  \\
+ \hal \int_{-m}^\ell \tep \left(  (\varphi' - \abs{\xi} \psi)^2 + (\psi'- \abs{\xi}\varphi)^2 + \frac{1}{3}(\psi' + \abs{\xi} \varphi)^2\right) 
\end{multline}
and 
\begin{equation}\label{J_def}
 J(\varphi,\psi) = \hal \int_{-m}^\ell \rho_0  (\varphi^2 + \psi^2),
\end{equation}
which are both well-defined on the space $H_0^1((-m,\ell)) \times H^1_0((-m,\ell))$.  Consider the set
\begin{equation}
 \mathcal{A} = \{ (\varphi,\psi)\in H_0^1((-m,\ell)) \times H^1_0((-m,\ell)) \;\vert\;  J(\varphi,\psi)=1  \}.
\end{equation}

We want to show that the infimum of $E(\varphi,\psi)$ over the set $\mathcal{A}$ is achieved and is negative, and that the minimizer solves the equations \eqref{s_coupled_1}--\eqref{s_coupled_2} along with the corresponding jump and boundary conditions.  Notice that the jump condition $\jump{\varphi} =\jump{\psi} =0$ holds trivially since $\varphi,\psi \in H_0^1((-m,\ell))$.  Also notice that by employing the identity $-2ab=(a-b)^2-(a^2+b^2)$ and the constraint on $J(\varphi,\psi)$ we may rewrite 
\begin{multline}\label{E_lower_bound}
E(\varphi,\psi) = -g \abs{\xi} + \frac{\sigma \abs{\xi}^2}{2}(\psi(0))^2 
+ \hal \int_{-m}^\ell (\td+  P'(\rho_0) \rho_0) (\psi' + \abs{\xi} \varphi)^2    
+  g\abs{\xi} \rho_0 (\varphi-\psi)^2 \\
+\hal \int_{-m}^\ell \tep \left(  (\varphi' - \abs{\xi} \psi)^2 + (\psi'- \abs{\xi}\varphi)^2 + \frac{1}{3}(\psi' + \abs{\xi} \varphi)^2\right) 
  \ge -g\abs{\xi}
\end{multline}
for any $(\varphi,\psi)\in \mathcal{A}$.   Recall that $\tep = s \ep(\rho_0)$, which is smooth when restricted to $(-m,0)$ and $(0,\ell)$ and bounded above and below by positive quantities for fixed $s>0$.  In order to emphasize the dependence on $s \in (0,\infty)$ we will sometimes write
\begin{equation}
 E(\varphi,\psi) = E(\varphi,\psi;s)
\end{equation}
and 
\begin{equation}\label{mu_def}
 \mu(s) := \inf_{(\varphi,\psi)\in \mathcal{A}}E(\varphi,\psi;s).
\end{equation}

As the first order of business we show that a minimizer exists.

\begin{prop}\label{min_exist}
$E$ achieves its infimum on $\mathcal{A}$.
\end{prop}
\begin{proof}
First note that  \eqref{E_lower_bound} shows that $E$ is bounded below on $\mathcal{A}$.  Let $(\varphi_n,\psi_n)\in\mathcal{A}$ be a minimizing sequence.  Then $\varphi_n$ and $\psi_n$ are bounded in $H^1_0((-m,\ell))$ and $\psi_n(0)$ is bounded in $\Rn{}$, so up to the extraction of a subsequence   $(\varphi_n,\psi_n) \rightharpoonup (\varphi,\psi)$ weakly in $H_0^1 \times H_0^1$, and $(\varphi_n,\psi_n) \rightarrow (\varphi,\psi)$ strongly in $L^2\times L^2$.   The compact embedding $H_0^1 \csubset H^{2/3} \hookrightarrow C^0$ implies that $\psi_n(0) \rightarrow \psi(0)$ as well.   Because of the quadratic structure of all the terms in the integrals defining $E$, weak lower semi-continuity and strong $L^2$ convergence imply that
\begin{equation}
 E(\varphi,\psi) \le \liminf_{n\rightarrow \infty} E(\varphi_n,\psi_n) = \inf_{\mathcal{A}} E.
\end{equation}
That $(\varphi,\psi)\in\mathcal{A}$ follows from the strong $L^2$ convergence.
\end{proof}

We now show that the minimizer constructed in the previous result satisfies Euler-Langrange equations equivalent to \eqref{s_coupled_1}--\eqref{s_coupled_2}.

\begin{prop}\label{e_l_eqns}
Let $(\varphi,\psi)\in \mathcal{A}$ be the minimizers of $E$ constructed in Proposition \ref{min_exist}.  Let $\mu := E(\varphi,\psi)$.  Then $(\varphi,\psi)$ are smooth when restricted to $(-m,0)$ or $(0,\ell)$ and satisfy
\begin{multline}\label{eigen_coupled_1}
\mu \rho_0 \varphi = -(\tep \varphi')' + \abs{\xi}^2  \left(4\tep/3 +   \td +  P'(\rho_0) \rho_0 \right) \varphi \\
+ \abs{\xi} \left[ \left(\td +\tep/3 + P'(\rho_0)\rho_0\right)\psi' +(\tep' - g\rho_0) \psi\right]  
\end{multline}
and 
\begin{multline}\label{eigen_coupled_2}
 \mu \rho_0 \psi =  -\left[ \left(  4\tep/3 + \td +  P'(\rho_0) \rho_0   \right) \psi'\right]' +   \tep \abs{\xi}^2  \psi \\
- \abs{\xi} \left[ \left(\left( \td +\tep/3 +  P'(\rho_0) \rho_0    \right) \varphi \right)'  + (g \rho_0-\tep' )\varphi \right]  
\end{multline}
along with the jump conditions
\begin{equation}
\jump{\varphi} = \jump{\psi}  =\jump{\tep (\varphi'-\abs{\xi} \psi)} = 0, 
\end{equation}
\begin{equation}
\jump{(\td+ \tep/3 + P'(\rho_0)\rho_0) (\psi' + \abs{\xi} \varphi  )}
+  \jump{\tep \left(\psi' - \abs{\xi} \varphi     \right)  }
  =\sigma \abs{\xi}^2 \psi(0).
\end{equation}
and the boundary conditions
$ \varphi(-m) =  \varphi(\ell) =  \psi(-m) =  \psi(\ell)=0.$
\end{prop}

\begin{proof}
 Fix $(\varphi_0,\psi_0)\in H_0^1((-m,\ell))\times H^1_0((-m,\ell))$.   Define
\begin{equation}
j(t,\tau)= J(\varphi+t\varphi_0 + \tau \varphi,\psi + t\psi_0 + \tau \psi) 
\end{equation}
and note that $ j(0,0) = 1$.  Moreover, $j$ is smooth,
\begin{equation}
 \frac{\partial j}{\partial t}(0,0) = \int_{-m}^\ell   \rho_0  ( \varphi_0 \varphi  + \psi_0 \psi   ), \text{ and }
 \frac{\partial j}{\partial \tau}(0,0) = \int_{-m}^\ell   \rho_0  (\varphi^2 + \psi^2)  =2.
\end{equation}
So, by the inverse function theorem, we can solve for $\tau = \tau(t)$ in a neighborhood of $0$ as a $C^1$ function of $t$ so that $\tau(0)=0$ and $j(t,\tau(t))=1$.  We may differentiate the last equation to find
\begin{equation}
 \frac{\partial j}{\partial t}(0,0) + \frac{\partial j}{\partial \tau}(0,0) \tau'(0) = 0,
\end{equation}
and hence that
\begin{equation}
 \tau'(0) = -\hal \frac{\partial j}{\partial t}(0,0) = -\hal \int_{-m}^\ell   \rho_0  ( \varphi_0 \varphi  + \psi_0 \psi ).
\end{equation}

Since $(\varphi,\psi)$ are minimizers over $\mathcal{A}$,  we may make variations with respect to $(\varphi_0,\psi_0)$ to find that
\begin{equation}
 0 = \left. \frac{d}{dt}\right\vert_{t=0} E(\varphi +t\varphi_0 + \tau(t) \varphi,\psi+t\psi_0 + \tau(t) \psi),
\end{equation}
which implies that 
\begin{multline}
0= \sigma \abs{\xi}^2 \psi(0)(\psi_0(0)+\tau'(0)\psi(0)) \\
+ \int_{-m}^\ell (\td + \tep/3+ P'(\rho_0) \rho_0) (\psi'+\abs{\xi} \varphi)(\psi_0' + \tau'(0) \psi' + \abs{\xi} \varphi_0 + \abs{\xi} \tau'(0) \varphi) \\
- \int_{-m}^\ell g \abs{\xi} \rho_0 (\psi(\varphi_0 + \tau'(0) \varphi) + \varphi(\psi_0 + \tau'(0) \psi)    ) \\
+ \int_{-m}^\ell \tep(\psi' - \abs{\xi} \varphi) (\psi_0' + \tau'(0) \psi' - \abs{\xi} \varphi_0 - \abs{\xi} \tau'(0) \varphi)    \\
+ \int_{-m}^\ell \tep(\varphi' - \abs{\xi} \psi)(\varphi_0' + \tau'(0) \varphi' - \abs{\xi} \psi_0 - \abs{\xi} \tau'(0) \psi)   .
\end{multline}
Rearranging and plugging in the value of $\tau'(0)$, we may rewrite this equation as
\begin{multline}\label{eigenvalue_form}
\sigma \abs{\xi}^2 \psi(0)\psi_0(0) + \int_{-m}^\ell (\td+\tep/3+P'(\rho_0) \rho_0) (\psi'+\abs{\xi} \varphi) ( \psi_0' + \abs{\xi} \varphi_0) - g \abs{\xi} \rho_0(\psi \varphi_0 + \varphi \psi_0) \\
+ \int_{-m}^\ell   \tep \left( (\psi' - \abs{\xi} \varphi) (\psi_0' - \abs{\xi} \varphi_0) + (\varphi' - \abs{\xi} \psi)(\varphi_0' - \abs{\xi}\psi_0) \right)
  = \mu \int_{-m}^\ell   \rho_0  ( \varphi_0 \varphi  + \psi_0 \psi )
\end{multline}
where the Lagrange multiplier (eigenvalue) is $\mu = E(\varphi,\psi).$  Since $\varphi_0$ and $\psi_0$ are independent, this gives rise to the pair of equations 
\begin{multline}\label{e_l_e_1}
\int_{-m}^\ell \tep \varphi' \varphi_0' 
+    \left(\td+4\tep/3+P'(\rho_0) \rho_0\right) \abs{\xi}^2 \varphi \varphi_0 
- \int_{-m}^\ell   \left( \tep \abs{\xi} \psi \varphi_0 \right)' \\
+\abs{\xi} \int_{-m}^\ell \left[   \left(\td+\tep/3+P'(\rho_0) \rho_0\right) \psi'  +(\tep'- g  \rho_0) \psi \right] \varphi_0 
  = \mu \int_{-m}^\ell   \rho_0   \varphi \varphi_0  
 \end{multline}
and
\begin{multline}\label{e_l_e_2}
\sigma \abs{\xi}^2 \psi(0) \psi_0(0) 
+ \int_{-m}^\ell \left[ \left(\td + 4\tep/3+P'(\rho_0) \rho_0 \right) \psi'  + \left(\td+\tep/3+P'(\rho_0) \rho_0 \right)     \abs{\xi} \varphi   \right] \psi_0'  \\
- \int_{-m}^\ell   \left( \tep \abs{\xi} \varphi \psi_0 \right)' 
+ \int_{-m}^\ell  \left[ \tep \abs{\xi}^2 \psi   + (\tep'   - g  \rho_0) \abs{\xi} \varphi  \right]   \psi_0
  = \mu \int_{-m}^\ell   \rho_0 \psi   \psi_0.
\end{multline}

By making variations with $\varphi_0, \psi_0$ compactly supported in either $(-m,0)$ or $(0,\ell)$, we find that $\varphi$ and $\psi$ satisfy the equations \eqref{eigen_coupled_1}--\eqref{eigen_coupled_2} in a weak sense in $(-m,0)$ and $(0,\ell)$.  Standard bootstrapping arguments then show that $(\varphi,\psi)$ are in $H^k((-m,0))$ (resp. $H^k((0,\ell))$)  for all $k\ge 0$ when restricted to $(-m,0)$ (resp. $(0,\ell)$), and hence the functions are smooth when restricted to either interval.  This implies that the equations are also classically satisfied on $(-m,0)$ and $(0,\ell)$.  Since $(\varphi,\psi)\in H^2$, the traces of the functions and their derivatives are well-defined at the endpoints $x_3 = -m,0,\ell$.  To show that the jump conditions are satisfied we make variations with respect to arbitrary $\varphi_0$, $\psi_0 \in C_c^\infty((-m,\ell))$.  Integrating  the terms in \eqref{e_l_e_1} with derivatives of $\varphi_0$ by parts and using that $\varphi$ solves \eqref{eigen_coupled_1} on $(-m,0)$ and $(0,\ell)$, we find that
\begin{equation}
\jump{\tep (\varphi'-\abs{\xi} \psi)} \varphi_0(0) = 0.  
\end{equation}
Since $\varphi_0(0)$ may be chosen arbitrarily, we deduce the jump condition $
\jump{\tep (\varphi'-\abs{\xi} \psi)}  = 0$.  Performing a similar integration by parts in \eqref{e_l_e_2} yields the jump condition
\begin{equation}
0=  \sigma \abs{\xi}^2 \psi(0)   - \jump{(\td+\tep/3+P'(\rho_0) \rho_0) (\psi'+\abs{\xi} \varphi)}   - \jump{\tep  (\psi' - \abs{\xi} \varphi)}.
 \end{equation}
The conditions $\jump{\varphi}=\jump{\psi }=0$  and $\varphi(-m) = \varphi(\ell) = \psi(-m) = \psi(\ell) = 0$ are satisfied trivially since $\varphi, \psi \in H_0^1((-m,\ell)) \hookrightarrow C_0^{0,1/2}((-m,\ell))$.  

\end{proof}

We now show that for $s$ sufficiently small, the infimum of $E$ over $\mathcal{A}$ is in fact negative.

\begin{prop}\label{neg_inf}
Suppose that  $0<\abs{\xi}^2 < g\jump{\rho_0}/\sigma$.  Then there exists $s_0>0$  depending on the quantities  $\rho_0^{\pm}, P_{\pm}, g, \ep_\pm, \sigma, m, \ell, \abs{\xi}$ so that for $s \le s_0$ it holds that $\mu(s)  < 0.$
\end{prop}
\begin{proof}
 Since both $E$ and $J$ are homogeneous of degree $2$ it suffices to show that 
\begin{equation}
 \inf_{(\varphi,\psi)\in H_0^1 \times H_0^1} \frac{E(\varphi,\psi)}{J(\varphi,\psi)} < 0,
\end{equation}
but since $J$ is positive definite, we may reduce to constructing any pair $(\varphi,\psi)\in H_0^1 \times H_0^1$ such that $E(\varphi,\psi) <0$.  We will  assume that $\varphi = -\psi'/\abs{\xi}$ so that the first integrand term in $E(\varphi,\psi)$ vanishes.  We must then construct $\psi \in H_0^2$ so that 
\begin{equation}
\tilde{E}(\psi) := E(-\psi'/\abs{\xi},\psi) = \frac{\sigma \abs{\xi}^2}{2}(\psi(0))^2 + \int_{-m}^\ell g \rho_0 \psi \psi' + \frac{\tep}{2} \left( \left( \frac{\psi''}{\abs{\xi}} + \abs{\xi} \psi  \right)^2  + 4(\psi')^2 \right)  <0. 
\end{equation}

We employ the identity $\psi \psi' = (\psi^2)'/2$, an integration by parts, and the fact that $\rho_0$ solves \eqref{steady_state} to write
\begin{multline}\label{n_i_1}
 \int_{-m}^\ell g \rho_0 \psi \psi' = \left[\frac{g \rho_0 \psi^2 }{2}  \right]_{0}^{\ell} - \hal \int_{0}^\ell g \rho_0' \psi^2 + \left[\frac{g \rho_0 \psi^2 }{2} \right]_{-m}^{0} - \hal \int_{-m}^0 g \rho_0' \psi^2 \\
=  -\frac{g (\psi(0))^2}{2} \jump{\rho_0} + \frac{g^2}{2} \int_{-m}^\ell \frac{\rho_0}{P'(\rho_0)} \psi^2.
\end{multline}
Notice that $\jump{\rho_0}= \rho^+_0 - \rho^-_0 >0$ so that the right hand side is not positive definite.

For $\alpha \ge 5$ we define the test function $\psi_\alpha \in H_0^2((-m,\ell))$ according to 
\begin{equation}
 \psi_\alpha(x_3) = 
\begin{cases}
  \left(1-\frac{x_3^2}{\ell^2} \right)^{\alpha/2}, & x_3 \in [0,\ell) \\
  \left(1+\frac{x_3^2}{m^2} \right)^{\alpha/2}, &  x_3\in (-m,0). 
\end{cases}
\end{equation}
Simple calculations then show that
\begin{equation}
 \int_{-m}^\ell (\psi_\alpha)^2 = \frac{\sqrt{\pi}(m+\ell) \Gamma(\alpha +1 )}{2 \Gamma(\alpha +3/2)} = o_\alpha(1),
\end{equation}
where $o_\alpha(1)$ is a quantity that vanishes as $\alpha \rightarrow \infty$, and that 
\begin{equation}
\int_{-m}^\ell \left( \left( \frac{\psi''}{\abs{\xi}} + \abs{\xi} \psi  \right)^2  + 4(\psi')^2 \right) \le C
\end{equation}
for a constant $C$ depending on $\alpha, m, \ell, \abs{\xi}$.  Combining these, we find that
\begin{equation}
 \tilde{E}(\psi_\alpha) \le \frac{\sigma \abs{\xi}^2- g\jump{\rho_0}}{2}   + o_\alpha(1) + s C
\end{equation}
for a constant $C$ depending on $\alpha$ as well as  $\rho_0^{\pm}, P_{\pm}, g, \ep_\pm, m, \ell, \abs{\xi}$.  Since $\sigma \abs{\xi}^2 < g \jump{\rho_0}$, we may then fix $\alpha$ sufficiently large so that the first two terms sum to something strictly negative.  Then there exists  $s_0>0$ depending on the various parameters so that for $s\le s_0$ it holds that $\tilde{E}(\psi_\alpha) < 0$, thereby proving the result.
\end{proof}

\begin{remark}\label{neg_inf_remark}
A simple extension of this argument yields a more quantitative bound that holds not only for $\abs{\xi}$ fixed, but also uniformly over intervals $0<a \le \abs{\xi}^2 \le b <g\jump{\rho_0}/\sigma$.  More precisely, there exist two constants $C_0, C_1>0$ depending on the parameters $\rho_0^{\pm},$ $P_{\pm},$ $g,$ $\ep_\pm,$ $\sigma,$ $m,$ $\ell,$ $a,$ $b$ so that $\mu(s) \le -C_0 + s C_1$ for all $\abs{\xi}^2 \in[a,b]$.
\end{remark}

The key to the argument presented in this Proposition was constructing a pair $(\varphi,\psi)$ so that
\begin{equation}
 \frac{\sigma \abs{\xi}^2 - g\jump{\rho_0}}{2} (\psi(0))^2 <0,
\end{equation}
which in particular required that $\psi(0)\neq 0$ and $\abs{\xi}^2 < g\jump{\rho_0}/\sigma$.  We can show that these properties are satisfied by the actual minimizers when $E(\varphi,\psi) <0$.

\begin{lem}\label{non_zero_origin}
 Suppose that $(\varphi,\psi)\in \mathcal{A}$ satisfy $E(\varphi,\psi)<0$.  Then $\psi(0)\neq 0$ and $\abs{\xi}^2 < g\jump{\rho_0}/\sigma$.
\end{lem}
\begin{proof}
A completion of the square allows us to write 
\begin{multline}
 P'(\rho_0)\rho_0 (\psi' + \abs{\xi} \varphi)^2 - 2 g \rho_0 \abs{\xi} \psi \varphi \\
= \left(\sqrt{P'(\rho_0)\rho_0}(\psi'+\abs{\xi} \varphi) - \frac{g\sqrt{\rho_0}}{\sqrt{P'(\rho_0)}}  \psi  \right)^2  + 2 g \rho_0 \psi \psi' - \frac{g^2 \rho_0}{P'(\rho_0)}  \psi^2.
\end{multline}
Integrating by parts as in \eqref{n_i_1}, we know that
\begin{equation}
  \int_{-m}^\ell 2g \rho_0 \psi \psi' -  \frac{g^2 \rho_0}{P'(\rho_0)} \psi^2=  -g  \jump{\rho_0} (\psi(0))^2.
\end{equation}
Combining these equalities, we can rewrite $E(\varphi,\psi)$ as 
\begin{multline}\label{n_z_o_1}
 E(\varphi,\psi) = \hal \int_{-m}^\ell (\td + \tep/3)(\psi'+\abs{\xi} \varphi)^2  + P'(\rho_0) \rho_0 \left((\psi'+\abs{\xi} \varphi) - \frac{g}{P'(\rho_0)}  \psi  \right)^2  \\
+  \hal\int_{-m}^\ell  \tep (  (\varphi'- \abs{\xi} \psi)^2 + (\psi'-\abs{\xi} \varphi)^2 )  +\frac{\sigma \abs{\xi}^2 - g\jump{\rho_0}}{2}(\psi(0))^2.
\end{multline}
From the non-negativity of the integrals, we deduce that if $E(\varphi,\psi)<0$, then $\psi(0) \neq 0$ and $\abs{\xi}^2 < g \jump{\rho_0}/\sigma$.
\end{proof}

The next result establishes continuity and monotonicity properties of the eigenvalue $\mu(s)$.

\begin{prop}\label{eigen_lip}
Let $\mu:(0,\infty)  \rightarrow \Rn{}$ be given by \eqref{mu_def}.
Then the following hold.
\begin{enumerate}
 \item $\mu \in C^{0,1}_{loc}((0,\infty))$, and in particular $\mu \in C^{0}((0,\infty))$.

\item There exists a positive constant $C_2 = C_2(\rho_0^{\pm}, P_{\pm}, g, \ep_\pm, \sigma, m, \ell)$ so that  
\begin{equation}\label{e_l_00}
 \mu(s) \ge -g\abs{\xi} + s C_2.
\end{equation}
 
\item $\mu(s)$ is strictly increasing.
\end{enumerate}

\end{prop}

\begin{proof}
Fix a compact interval $Q = [a,b]  \csubset (0,\infty)$, and fix any pair $(\varphi_0,\psi_0)\in \mathcal{A}$.  We may decompose $E$ according to 
\begin{equation}\label{e_l_1}
E(\varphi,\psi;s) = E_0(\varphi,\psi) + s E_1(\varphi,\psi)  
\end{equation}
for 
\begin{equation}
 E_0(\varphi,\psi) := \frac{\sigma \abs{\xi}^2}{2}(\psi(0))^2 + \hal \int_{-m}^\ell P'(\rho_0) \rho_0 (\psi'+\abs{\xi} \varphi)^2 - 2 g\abs{\xi} \rho_0 \varphi \psi
\end{equation}
and
\begin{equation}
 E_1(\varphi,\psi) := \hal \int_{-m}^\ell (\delta_0+ \ep_0/3) (\psi'+\abs{\xi} \varphi)^2 + \ep_0 \left((\varphi'- \abs{\xi} \psi)^2 + (\psi'-\abs{\xi} \varphi)^2 \right) \ge 0.
\end{equation}
The non-negativity of $E_1$ implies that $E$ is non-decreasing in $s$  with $(\varphi,\psi)\in \mathcal{A}$ kept fixed.  
 
Now, by Proposition \ref{min_exist}, for each $s \in (0,\infty) $ we can find a pair $(\varphi_s, \psi_s)\in\mathcal{A}$ so that
\begin{equation}
 E(\varphi_s, \psi_s; s) = \inf_{(\varphi,\psi)\in \mathcal{A}} E(\varphi,\psi; s) = \mu(s).
\end{equation}
We deduce from the non-negativity of $E_1$, the minimality of $(\varphi_s, \psi_s)$, and the equality \eqref{E_lower_bound} that
\begin{equation}
 E(\varphi_0,\psi_0;b) \ge E(\varphi_0,\psi_0;s) \ge  E(\varphi_s, \psi_s;s) \ge  s E_1(\varphi_s, \psi_s) - g \abs{\xi}
\end{equation}
for all $s \in Q$.  This implies that there exists a constant $0<K = K(a,b, \varphi_0, \psi_0, g, \abs{\xi}) < \infty$ so that
\begin{equation}\label{e_l_2}
 \sup_{s \in Q} E_1(\varphi_s,\psi_s)   \le K.
\end{equation}

Let $s_i   \in Q$ for $i=1,2$.  Using the minimality of $(\varphi_{s_1},\psi_{s_1})$ compared to $(\varphi_{s_2}, \psi_{s_2})$, we know that
\begin{equation}
 \mu(s_1) = E(\varphi_{s_1},\psi_{s_1};s_1) \le E(\varphi_{s_2},\psi_{s_2};s_1),
\end{equation}
but from our decomposition \eqref{e_l_1}, we may bound
\begin{multline}
E(\varphi_{s_2}, \psi_{s_2};s_1) 
\le  E(\varphi_{s_2},\psi_{s_2};s_2) + \abs{s_1 - s_2} E_1(\varphi_{s_2}, \psi_{s_2})   \\
= \mu(s_2) + \abs{s_1 - s_2} E_1(\varphi_{s_2}, \psi_{s_2}) . 
\end{multline}
Chaining these two inequalities together and employing \eqref{e_l_2}, we find that
\begin{equation}
\mu(s_1)
\le \mu(s_2) + K\abs{s_1- s_2}.
\end{equation}
Reversing the role of the indices $1$ and $2$ in the derivation of this inequality gives the same bound with the indices switched.  We deduce that 
\begin{equation}
 \abs{\mu(s_1) - \mu(s_2)} \le K \abs{s_1-s_2},
\end{equation}
which proves the first assertion.

To prove \eqref{e_l_00} we note that equality \eqref{E_lower_bound} and the non-negativity of $E_1$ imply that
\begin{equation}
 \mu(s) \ge -g \abs{\xi} + s \inf_{(\varphi,\psi)\in\mathcal{A}}    E_1(\varphi,\psi).  
\end{equation}
It is a simple matter to see that this infimum, which we call the constant $C_2$, is positive.  Finally, to prove the third assertion, note that if  $0 < s_1 < s_2< \infty$,  then the decomposition \eqref{e_l_1} implies that
\begin{equation}
 \mu(s_1) = E(\varphi_{s_1}, \psi_{s_1};s_1) \le E(\varphi_{s_2}, \psi_{s_2};s_1) \le  E(\varphi_{s_2}, \psi_{s_2}; s_2) = \mu(s_2).
\end{equation}
This shows that $\mu$ is non-decreasing in $s$.  Now suppose  by way of contradiction that $\mu(s_1) = \mu(s_2)$.  Then the previous inequality implies that  
\begin{equation}
 s_1 E_1(\varphi_{s_2}, \psi_{s_2}) = s_2 E_1(\varphi_{s_2}, \psi_{s_2}),
\end{equation}
which means that $E_1(\varphi_{s_2}, \psi_{s_2})=0$.  This in turn forces $\varphi_{s_2} = \psi_{s_2} = 0$, which contradicts the fact that $(\varphi_{s_2}, \psi_{s_2}) \in  \mathcal{A}$.  Hence equality cannot be achieved, and $\mu$ is strictly increasing in $s$. 

\end{proof}

Now we know that when $0< \abs{\xi}^2 < g\jump{\rho_0}/\sigma$, the eigenvalue $\mu(s)$ is a continuous function.  We can then define the open set
\begin{equation}
 \mathcal{S} = \mu^{-1}((-\infty,0)) \subset (0,\infty),
\end{equation}
on which we can calculate $\lambda = \sqrt{-\mu} >0$.  Note that $\mathcal{S}$ is non-empty by Proposition \ref{neg_inf}.

We can now state a result giving the existence of solutions to \eqref{s_coupled_1}--\eqref{s_coupled_2} for these values of $\abs{\xi},s$.  To emphasize the dependence on the parameters, we write
\begin{equation}
 \varphi = \varphi_{s}(\abs{\xi},x_3),  \psi= \psi_{s}(\abs{\xi},x_3), \text{ and } \lambda = \lambda(\abs{\xi},s).
\end{equation}

\begin{prop}\label{coupled_solution}
For each $s\in \mathcal{S}$  and $0<\abs{\xi}^2 < g\jump{\rho_0}/\sigma$ there exists a solution $\varphi_{s}(\abs{\xi},x_3),$  $\psi_{s}(\abs{\xi},x_3)$ with $\lambda = \lambda(\abs{\xi},s)>0$ to the problem \eqref{s_coupled_1}--\eqref{s_coupled_2} along with the corresponding jump and boundary conditions. For these solutions $\psi_{s}(\abs{\xi},0) \neq 0$ and the solutions are smooth when restricted to either $(-m,0)$ or $(0,\ell)$.
\end{prop}

\begin{proof}
Let $(\varphi_{s}(\abs{\xi},\cdot),\psi_{s}(\abs{\xi},\cdot)) \in \mathcal{A}$ be the solutions to \eqref{eigen_coupled_1}--\eqref{eigen_coupled_2} constructed in Proposition \ref{e_l_eqns}.  Since $s \in \mathcal{S}$ we may write $\mu = -\lambda^2$ for $\lambda>0$, which means that the pair $(\varphi_{s}(\abs{\xi},\cdot),\psi_{s}(\abs{\xi},\cdot))$ solve the problem \eqref{s_coupled_1}--\eqref{s_coupled_2}.  The fact that $\psi_{s}(\abs{\xi},0) \neq 0$ follows from Lemma \ref{non_zero_origin}.
\end{proof}

In order for these solutions to give rise to solutions to the original problem, we must be able to find $s \in \mathcal{S}$ so that $s= \lambda(\abs{\xi},s)$. It turns out that the set $\mathcal{S}$ is sufficiently large to accomplish this.  

\begin{thm}\label{inversion}
There exists a unique $s \in \mathcal{S}$ so that $\lambda(\abs{\xi},s)=\sqrt{-\mu(s)}>0$ and 
\begin{equation}\label{inv_0}
 s= \lambda(\abs{\xi},s).
\end{equation}
\end{thm}

\begin{proof}
According to Remark \ref{neg_inf_remark}, we know that $\mu(s) \le -C_0 + s C_1$.  Moreover, the lower bound \eqref{e_l_00} implies that $\mu(s)\rightarrow +\infty$ as $s \rightarrow \infty$.  Since $\mu$ is continuous and strictly increasing, there exists $s_*\in(0,\infty)$ so that 
\begin{equation}
 \mathcal{S} = \mu^{-1}((-\infty,0)) = (0,s_*).
\end{equation}
Since $\mu < 0$ on $\mathcal{S}$, we may define $\lambda=\sqrt{-\mu}$ there.  Now define the function $\Phi: (0,s_*)  \rightarrow (0,\infty)$ according to 
\begin{equation}
 \Phi(s) = s/\lambda(\abs{\xi},s).
\end{equation}
It is a simple matter to check that the continuity and monotonicity properties of $\mu$ are inherited by $\Phi$, i.e. $\Phi$ is continuous and  strictly increasing in $s$.  Also, $\lim_{s \rightarrow 0}\Phi(s) = 0$ and $\lim_{s\rightarrow s_*}\Phi(s)=+\infty$.  Then by the intermediate value theorem, there exists $s \in (0,s_*)$ so that $\Phi(s) =1$, i.e. 
$s = \lambda(\abs{\xi},s).$  This $s$ is unique since $\Phi$ is strictly increasing.
\end{proof}

We may now use Theorem \ref{inversion} to think of $s = s(\abs{\xi})$ since for each fixed $0< \abs{\xi}^2< g\jump{\rho_0}/\sigma$ we can uniquely find $s\in\mathcal{S}$ so that \eqref{inv_0} holds.  As such we may also write $\lambda = \lambda(\abs{\xi})$ from now on.

Using this new notation and the solutions to \eqref{s_coupled_1}--\eqref{s_coupled_2} given by Proposition \ref{coupled_solution}, we can construct solutions to the system system \eqref{w_1_equation}--\eqref{w_3_equation} as well.

\begin{proof}[Proof of Theorem \ref{w_soln_2}]
We may find a rotation operator $R \in SO(2)$ so that $R \xi = (\abs{\xi},0)$.   For $s = s(\abs{\xi})$ given by Theorem \ref{inversion}, define $(\varphi(\xi,x_3),\theta(\xi,x_3)) = R^{-1} (\varphi_{s}(\abs{\xi},x_3),0)$ and $\psi(\xi,x_3) = \psi_{s}(\abs{\xi},x_3)$, where the functions $\varphi_{s}(\abs{\xi},x_3)$ and $\psi_{s}(\abs{\xi},x_3)$ are the solutions from Proposition \ref{coupled_solution}.  This gives a solution to \eqref{w_1_equation}--\eqref{w_3_equation}.  The equivariance in $\xi$ follows from the definition.
\end{proof}

\subsection{Behavior of the solutions with respect to $\xi$}

In this section we shall study the behavior of the solutions from Theorem \ref{w_soln_2} in terms of $\xi$.  We assume throughout that $\abs{\xi} \in (0,\abs{\xi}_c)$ with $\abs{\xi}_c = \sqrt{g\jump{\rho_0}/\sigma}$.  The results are primarily needed in the non-periodic case, when there is a continuum of spatial frequencies in $(0,\abs{\xi}_c)$.

The first result shows that $\lambda$ is a bounded, continuous function of $\abs{\xi}$.

\begin{prop}\label{lambda_cont}
The function $\lambda:(0,\abs{\xi}_c) \rightarrow (0,\infty)$ is bounded, continuous, and satisfies
\begin{equation}
 \lim_{\abs{\xi}\rightarrow 0} \lambda(\abs{\xi}) = \lim_{\abs{\xi}\rightarrow \abs{\xi}_c} \lambda(\abs{\xi}) = 0.
\end{equation} 
\end{prop}
\begin{proof}
We begin by proving the continuity claim.  Since $\lambda = \sqrt{-\mu}$ it suffices to prove the continuity of $\mu = \mu(\abs{\xi})$.  By Proposition \ref{e_l_eqns}, for every $\abs{\xi}\in(0,\abs{\xi}_c)$ there exist functions $(\varphi_{\abs{\xi}},\psi_{\abs{\xi}})\in \mathcal{A}$ satisfying \eqref{eigen_coupled_1}--\eqref{eigen_coupled_2} so that $\mu(\abs{\xi}) = E(\varphi_{\abs{\xi}}, \psi_{\abs{\xi}})$. 
We have that $\mu(\abs{\xi}) < 0$, which, when combined with  \eqref{E_lower_bound}, yields the bound
\begin{equation}\label{l_c_1}
-g \abs{\xi} +  s(\abs{\xi})   \int_{-m}^\ell \frac{\ep_0}{2} \left((\varphi_{\abs{\xi}}' - \abs{\xi} \psi_{\abs{\xi}})^2 + (\psi_{\abs{\xi}}'- \varphi_{\abs{\xi}})^2 ) \le \mu(\abs{\xi} \right) < 0
\end{equation}
for all $\abs{\xi}$.

Now suppose $\abs{\xi}_n \in (0,g\jump{\rho_0}/\sigma)$ is a sequence so that $\abs{\xi}_n \rightarrow \abs{\xi} \in(0,g\jump{\rho_0}/\sigma)$.  We may assume without loss of generality that $\abs{\xi}_n\in[\abs{\xi}/2,(\abs{\xi}+\abs{\xi}_c)/2]$  if $\sigma>0$ or  $\abs{\xi}_n\in[\abs{\xi}/2,2\abs{\xi}]$ if $\sigma=0$.  In order to make use of the bound \eqref{l_c_1} we must show that $s(\abs{\xi}_n)$ is bounded uniformly from below as $n \rightarrow \infty$.  By Remark \ref{neg_inf_remark}, there exist positive constants $C_0,C_1$ so that $\mu(\abs{\xi}_n) \le -C_0 + s(\abs{\xi}_n) C_1$, but $-\mu(\abs{\xi}_n) =\lambda^2(\abs{\xi}_n) = s^2(\abs{\xi}_n)$, so 
\begin{equation}
 0 \le s^2(\abs{\xi}_n) + C_1 s(\abs{\xi}_n)  - C_0 
\end{equation}
and hence $s(\abs{\xi}_n)$ is bounded below by a positive constant.  Then \eqref{l_c_1} and the fact that $(\varphi_{\abs{\xi}_n},\psi_{\abs{\xi}_n})\in \mathcal{A}$ imply that $\varphi_{\abs{\xi}_n}$ and $\psi_{\abs{\xi}_n}$ are uniformly bounded in $H^1((-m,\ell))$.  Plugging into the ODE \eqref{eigen_coupled_1}--\eqref{eigen_coupled_2} in the intervals $(-m,0)$ and $(0,\ell)$ separately, we find that $\varphi_{\abs{\xi}_n}$ and $w_{\abs{\xi}_n}$ are uniformly bounded in $H^2((-m,0))$ and $H^2((0,\ell))$.  So, up to the extraction of a subsequence we have that
\begin{equation}
(\varphi_{\abs{\xi}_n},\psi_{\abs{\xi}_n}) \rightarrow (\varphi_{\abs{\xi}},\psi_{\abs{\xi}}) \text{ strongly in } H^1((-m,0)) \text{ and } H^1((0,\ell)).
\end{equation}
This implies that along the subsequence
\begin{equation}
 \mu(\abs{\xi}_n) = E(\varphi_{\abs{\xi}_n}, \psi_{\abs{\xi}_n}) \rightarrow E(\varphi_{\abs{\xi}},\psi_{\abs{\xi}}) = \mu(\abs{\xi}).
\end{equation}
Since this must hold for any such extracted subsequence, we deduce that $\mu(\abs{\xi}_n) \rightarrow \mu(\abs{\xi})$ for the original sequence $\abs{\xi}_n$ as well, and hence $\mu$ is continuous.

We now derive the limits as $\abs{\xi} \rightarrow 0,\abs{\xi}_c$.  By \eqref{l_c_1}, 
$0 \le  \lambda^2(\abs{\xi}) \le g \abs{\xi},$ which establishes that $\lim_{\abs{\xi} \rightarrow 0} \lambda(\abs{\xi})=0$.  By \eqref{E_lower_bound} we know that
\begin{equation}
 (\psi_{\abs{\xi}}(0))^2 \le \frac{2 g}{\sigma \abs{\xi}},
\end{equation}
but by \eqref{n_z_o_1} we also know that
\begin{equation}
 \lambda^2(\abs{\xi}) \le \frac{g \jump{\rho_0} - \sigma \abs{\xi}^2}{2} (\psi_{\abs{\xi}}(0))^2.
\end{equation}
Chaining the two inequalities together then shows that $\lim_{\abs{\xi} \rightarrow \abs{\xi}_c} \lambda(\abs{\xi})=0$.

\end{proof}

\begin{remark}
A trivial consequence of this result is that the supremum of $\lambda$ is achieved.  We denote the supremum in the non-periodic case by $\Lambda$ (see \eqref{max_def}) and in the periodic case by $\Lambda_L$ (see \eqref{max_def_per}).  
\end{remark}

The next result provides an estimate for the $H^k$ norm of the solutions $(\varphi,\theta,\psi)$ constructed in Theorem \ref{w_soln_2}, which will be useful later when such solutions are integrated in a Fourier synthesis.

\begin{lem}\label{sobolev_bounds}
Suppose $0<a<b<\abs{\xi}_c$ and that $\abs{\xi} \in [a,b]$.  Let $(\varphi, \theta, \psi)$ be the solutions constructed in Theorem \ref{w_soln_2}.  Then for each $k \ge 0$ there exists a constant $A_k >0$ depending on the parameters  $a,b,\rho^\pm_0, P_{\pm}, g,  \ep_\pm, \delta_\pm, \sigma, m, \ell$, 
\begin{multline}\label{s_b_0}
 \norm{\varphi( \xi,\cdot)}_{H^k((-m,0))} + \norm{\theta( \xi,\cdot)}_{H^k((-m,0))} + \norm{\psi(\xi,\cdot)}_{H^k((-m,0))}  \\
+ \norm{\varphi( \xi,\cdot )}_{H^k((0,\ell))} + \norm{\theta( \xi,\cdot )}_{H^k((0,\ell))} + \norm{\psi( \xi,\cdot )}_{H^k((0,\ell))} \le A_k.
\end{multline}
Also, there exists a $B_0>0$ depending on the same parameters so that 
\begin{equation}\label{s_b_00}
 \norm{\sqrt{\varphi^2( \xi,\cdot)  + \theta^2( \xi,\cdot) + \psi^2( \xi,\cdot) }}_{L^2((-m,\ell))} \ge B_0.
\end{equation}

\end{lem}
\begin{proof}
Since the solutions in Theorem \ref{w_soln_2} are constructed from  rotations of the solutions constructed in Proposition \ref{coupled_solution}, it suffices to prove
\begin{multline} 
 \norm{\varphi( \abs{\xi},\cdot)}_{H^k((-m,0))} +  \norm{\psi(\abs{\xi},\cdot)}_{H^k((-m,0))}  \\
+ \norm{\varphi( \abs{\xi},\cdot )}_{H^k((0,\ell))}   + \norm{\psi( \abs{\xi},\cdot )}_{H^k((0,\ell))} \le A_k
\end{multline}
for the solutions $\varphi = \varphi(\abs{\xi},x_3)$, $\varphi = \varphi(\abs{\xi},x_3)$ constructed in the theorem.  For simplicity we will prove an estimate of the $H^k$ norms only on the interval $(0,\ell)$.  A bound on $(-m,0)$ follows similarly, and the result follows by adding the two.  Recall that $\rho_0$ and $P'(\rho_0)$ are smooth on each interval $(0,\ell)$ and $(-m,0)$ and bounded above and below.  

We proceed by induction on $k$.  For $k=0$ the fact that $(\varphi(\abs{\xi},\cdot),\psi(\abs{\xi},\cdot))\in \mathcal{A}$ implies that there is a constant $A_0>0$ depending on the various parameters so that
\begin{equation}
\norm{\varphi(\abs{\xi},\cdot)}_{L^2((0,\ell))} + \norm{\psi(\abs{\xi},\cdot)}_{L^2((0,\ell))} \le A_0.
\end{equation}  
Suppose now that the bound holds some $k\ge 0$, i.e.
\begin{equation}
 \norm{\varphi(\abs{\xi},\cdot)}_{H^k((0,\ell))} + \norm{\psi(\abs{\xi},\cdot)}_{H^k((0,\ell))} \le A_k.
\end{equation}
By Proposition \ref{lambda_cont}, $\lambda(\abs{\xi}) = s(\abs{\xi})$ is bounded above and below by positive quantities as functions of $\abs{\xi}$.  Then by differentiating the equations \eqref{s_coupled_1}--\eqref{s_coupled_2} we have that there exists a constant $C>0$ depending on the various parameters so that
\begin{multline}
 \norm{\varphi(\abs{\xi},\cdot)}_{H^{k+1}((0,\ell))} + \norm{\psi(\abs{\xi},\cdot)}_{H^{k+1}((0,\ell))}  \\
\le C(\norm{\varphi(\abs{\xi},\cdot)}_{H^k((0,\ell))} + \norm{\psi(\abs{\xi},\cdot)}_{H^k((0,\ell))} ) \le C A_k := A_{K+1}.
\end{multline}
Then the bound holds for $k+1$, and so by induction the bound holds for all $k\ge 0$.

To prove \eqref{s_b_00} we again utilize the fact that   $(\varphi(\abs{\xi},\cdot),\psi(\abs{\xi},\cdot))\in \mathcal{A}$.  Since $\rho_0$ is bounded above and below, the bound follows.

\end{proof}

\subsection{Solutions to \eqref{linearized_1}--\eqref{linearized_2}}\label{growing_section}

In this section we will construct growing  solutions to \eqref{linearized_1}--\eqref{linearized_2} by using the solutions to \eqref{w_1_equation}--\eqref{w_3_equation} constructed in Theorem \ref{w_soln_2}.  In the periodic case this can only be done when $L$ satisfies \eqref{L_large}, but the construction is essentially trivial since normal mode solutions are in $L^2(\Omega)$.  In the non-periodic case, we must resort to a Fourier synthesis of the normal modes in order to produce $L^2(\Omega)$ solutions.

We begin by defining some terms.  For a function $f\in L^2(\Omega)$, we define the horizontal Fourier transform in the non-periodic case via
\begin{equation}\label{hft}
 \hat{f}(\xi_1,\xi_2,x_3) = \int_{\Rn{2}} f(x_1,x_2,x_3) e^{-i(x_1 \xi_1 + x_2 \xi_2)}dx_1 dx_2
\end{equation}
for $\xi \in \Rn{2}$.  In the periodic case the integral over $\Rn{2}$ must be replaced with an integral over $(2\pi L\mathbb{T})^2$ for $\xi \in (L^{-1}\mathbb{Z})^2$.  In the non-periodic case, by the Fubini and Parseval theorems, we have that $\hat{f} \in L^2(\Omega)$ and  
\begin{equation}\label{parseval}
\int_\Omega \abs{f(x)}^2 dx = \frac{1}{4\pi^2} \int_\Omega \abs{\hat{f}(\xi,x_3)}^2 d\xi dx_3.
\end{equation}
The periodic case replaces $4\pi^2$ with $4\pi^2 L^2$ and the integral with a sum over $(L^{-1}\mathbb{Z})^2$ on the right hand side.

We now define the piecewise Sobolev spaces.  For a function $f$ defined on $\Omega$ we write $f_+$ for the restriction to $\Omega_+$ and $f_-$ for the restriction to $\Omega_-$.  For $k \in \mathbb{N}$, define the piecewise Sobolev space of order $k$ by 
\begin{equation}\label{sob_def}
H^k(\Omega) = \{f \;\vert\; f_+ \in H^k(\Omega_+), f_- \in H^k(\Omega_-)  \}
\end{equation} 
endowed with the norm $\norm{f}_{H^k}^2 = \norm{f}_{H^k(\Omega_+)}^2 + \norm{f}_{H^s(\Omega_-)}^2$.   Writing $I_- = (-m,0)$ and $I_+ = (0,\ell)$, we can take the norms to be given as
\begin{equation}
 \norm{f}_{H^k(\Omega_\pm)}^2 :=   \sum_{j=0}^k  \int_{\Rn{2}} (1+\abs{\xi}^2)^{k-j} \norm{\partial_{x_3}^j \hat{f}_\pm(\xi,\cdot) }^2_{L^2(I_\pm)} d\xi 
\end{equation}
in the non-periodic case; for the periodic case we replace  the integral over $\Rn{2}$ with a sum over $(L^{-1}\mathbb{Z})^2$ on the right hand side.  The main difference between the piecewise Sobolev space $H^k(\Omega)$ and the usual Sobolev space is that we do not require functions in the piecewise space to have weak derivatives across the interface $\{x_3=0\}$.

The $2\pi L$ periodic growing mode solutions may now be constructed.

\begin{proof}[Proof of Theorem \ref{growing_mode_soln_periodic}]
It is clear that $\eta,v,q$ defined in this way are solutions to \eqref{linearized_1}--\eqref{linearized_2}.  That they are real-valued follows from the equivariance in $\xi$ stated in Theorem \ref{w_soln_2}.  The solutions are in $H^k(\Omega)$ at $t=0$ because of Lemma \ref{sobolev_bounds}.  The growth in time stated in \eqref{gmsp_0} follows from the definition of $\eta,v,q$.
\end{proof}

In the non-periodic case the exponentials $e^{i x' \cdot \xi}$ are not in $L^2(\Omega)$, so we must utilize a Fourier synthesis.  The tradeoff for utilizing such a synthesis is that the growth rate is not exactly $e^{\Lambda t}$, but can be made arbitrarily close to it.

\begin{proof}[Proof of Theorem \ref{growing_mode_soln}]
For each fixed $\xi\in \Rn{2}$ so that $\abs{\xi} \in(0,\abs{\xi}_c)$,  
\begin{equation}
 \eta(x,t) =  f(\abs{\xi}) \hat{w}(\xi,x_3) e^{\lambda(\abs{\xi}) t} e^{i x'\cdot \xi},
\end{equation}
\begin{equation}
 v(x,t) = \lambda(\abs{\xi}) f(\abs{\xi}) \hat{w}(\xi,x_3) e^{\lambda(\abs{\xi}) t} e^{i x'\cdot \xi}, \text{ and }
\end{equation}
\begin{equation}
 q(x,t) 
= -\rho_0(x_3) f(\abs{\xi}) (\xi_1 \varphi(\xi,x_3) +  \xi_2 \theta(\xi,x_3) + \partial_3 \psi(\xi,x_3)) e^{\lambda(\abs{\xi}) t} e^{i x'\cdot \xi}
\end{equation}
constitute a solution to \eqref{linearized_1}--\eqref{linearized_2}.  Since  $\supp(f)\csubset (0,\abs{\xi}_c),$ Lemma \ref{sobolev_bounds} implies that 
\begin{equation}
\sup_{\xi \in \supp(f)} \norm{ \partial_{x_3}^k \hat{w}(\xi,\cdot)}_{L^\infty} < \infty \text{ for all } k\in \mathbb{N}.
\end{equation}
These bounds, the definition of $\Lambda$, and the dominated convergence theorem imply that the Fourier synthesis of these solutions given by \eqref{g_m_s_1}--\eqref{g_m_s_2} is also a solution that is smooth when restricted to $\Omega_\pm$.  The Fourier synthesis is real-valued because $f(\abs{\xi})$ is real-valued and radial and because of the equivariance in $\xi$ given in Theorem \ref{w_soln_2}.  This equivariance in $\xi$ also implies the equivariance of $\eta, v, q$ written in \eqref{g_m_s_00}.    

The bound \eqref{g_m_s_0} follows by applying Lemma \ref{sobolev_bounds}  with arbitrary $k \ge 0$ and utilizing the fact that $f$ is compactly supported.  The compact support of $f$ also implies that $\lambda_0(f)>0$, so that $\lambda_0(f) \le \lambda(\abs{\xi}) \le \Lambda$ for $\abs{\xi} \in \supp(f)$.  This then yields the bounds \eqref{g_m_s_3}.
\end{proof}

\section{Growth of solutions to the linearized problem}

\subsection{Preliminary estimates}

In this section we will prove estimates for the growth in time of arbitrary solutions to \eqref{linearized_1}--\eqref{linearized_2} in terms of the largest growing mode: $\Lambda$ in the non-periodic case and  $\Lambda_L$ in the periodic case, defined by \eqref{max_def} and \eqref{max_def_per} respectively.  To this end, we suppose that $\eta,v,q$ are real-valued solutions to \eqref{linearized_1}--\eqref{linearized_2} along with the corresponding jump and boundary conditions (of course, by linearity, we may also handle complex solutions by taking the real and complex parts and proceeding with an analysis of each part).  

It will be convenient to work with a second-order formulation of the equations.  To arrive at this, we differentiate the third equation in time and eliminate the $q$ and $\eta$ terms using the other equations.  This yields the equation
\begin{multline}\label{second_order}
 \rho_0 \partial_{tt} v - \nab(P'(\rho_0) \rho_0 \diverge{v}) + g \rho_0 \nab v_3 -  g \rho_0 \diverge{v} e_3 \\
= \diverge\left( \ep_0 \left( D\dt v + D \dt v^T - \frac{2}{3} (\diverge{\dt v}) I\right) + \delta_0 (\diverge{\dt v}) I    \right) 
\end{multline}
coupled to the jump conditions
\begin{equation}
 \jump{\dt v} =0
\end{equation}
and
\begin{equation}
\jump{ ( P'(\rho_0) \rho_0 \diverge{v} )   I  + \ep_0 (D \dt v + D \dt v^T) + (\delta_0-2\ep_0/3) \diverge{\dt v} I   } e_3  = -\sigma \Delta_{x_1,x_2} v_3 e_3 .
\end{equation}
The function $\dt v$ also satisfies $\dt v(x_1,x_2,-m,t) = \dt v(x_1,x_2,\ell,t) =0$ at the upper and lower boundaries. The initial data for $\dt v(0)$ is given in terms of the initial data $q(0), v(0),$ and $\eta(0)$ via the third linear equation, i.e. $\dt v(0)$ satisfies
\begin{multline}
  \rho_0 \dt v(0) = - g q(0) e_3 - g \rho_0 \nab \eta_3(0) \\
+\diverge\left( \ep_0 \left( D v(0) + D  v(0)^T - \frac{2}{3} (\diverge{ v(0)}) I\right) + \delta_0 (\diverge{ v(0)}) I    \right). 
\end{multline}

Our first result gives an energy and its evolution equation for solutions to the second-order problem.

\begin{lem}\label{lin_en_evolve}
Let $v$ solve \eqref{second_order} and the corresponding jump and boundary conditions.  Then in the non-periodic case, 
\begin{multline}
 \dt \int_\Omega \rho_0 \frac{\abs{\dt v}^2}{2} + \frac{P'(\rho_0) \rho_0}{2}\abs{ \diverge{v} - \frac{g}{P'(\rho_0)} v_3 }^2 + \int_\Omega \frac{\ep_0}{2} \abs{D \dt v + D \dt v^T - \frac{2}{3} (\diverge{\dt v})I  }^2 \\
+\int_\Omega \delta_0 \abs{\diverge{\dt v}}^2
= \dt \int_{\Rn{2}} \frac{g \jump{\rho_0}}{2} \abs{v_3}^2 - \frac{\sigma}{2}\abs{\nab_{x_1,x_2} v_3}^2.
\end{multline}
In the periodic case, the same equation holds with the integral over $\Rn{2}$ replaced with an integral over $(2\pi L \mathbb{T})^2$.
\end{lem}
\begin{proof}
We will prove the result in the non-periodic case.  The periodic case follows similarly.  Recall that $\Omega_+ = \Rn{2} \times(0,\ell)$.  Take the dot product of \eqref{second_order} with $\dt v(t)$ and integrate over $\Omega_+$.  After integrating by parts and utilizing \eqref{steady_state}, we get
\begin{multline}
 \int_{\Omega_+}  \rho_0 \dt v \cdot \partial_{tt}v + P'(\rho_0)\rho_0 (\diverge{v}) (\diverge{\dt v}) -g\rho_0(v_3 \diverge{\dt v} + \dt v_3 \diverge{v}) + \frac{g^2 \rho_0}{P'(\rho_0)} v_3 \dt v_3
\\
+ \int_{\Omega_+} \frac{\ep_0}{2} \abs{D \dt v + D \dt v^T - \frac{2}{3} (\diverge{\dt v})I  }^2 
 + \int_{\Omega_+} \delta_0 \abs{\diverge{\dt v}}^2 \\
= \int_{\Rn{2}} g \rho^+_0  v_3 \dt v_3 
- \int_{\Rn{2}}  P_+'(\rho^+_0)\rho^+_0  \diverge{v} \dt v_3
- \int_{\Rn{2}} T e_3 \cdot \dt v
\end{multline}
where we have written 
\begin{equation}
 T = ( P'(\rho_0) \rho_0 \diverge{v} )   I  + \ep_0 \left(D \dt v + D \dt v^T -\frac{2}{3} (\diverge{\dt v})I \right) + \delta_0 \diverge{\dt v} I.
\end{equation}
We may pull time derivatives out of the first integrals on each side of the equation to arrive at the equality
\begin{multline}
 \dt \int_{\Omega_+} \rho_0 \frac{\abs{\dt v}^2}{2} + \frac{P'(\rho_0) \rho_0}{2} \abs{\diverge{v} - \frac{g}{P'(\rho_0)} v_3 }^2 + \int_{\Omega_+} \frac{\ep_0}{2} \abs{D \dt v + D \dt v^T - \frac{2}{3} (\diverge{\dt v})I  }^2 \\
 + \int_{\Omega_+} \delta_0 \abs{\diverge{\dt v}}^2 
= \dt \int_{\Rn{2}} g \rho_0^+ \frac{\abs{v_3}^2}{2} 
- \int_{\Rn{2}} T e_3 \cdot \dt v.
\end{multline}
A similar result holds on $\Omega_-= \Rn{2} \times(-m,0)$ with the opposite sign on the right hand side.  Adding the two together yields
\begin{multline}
 \dt \int_{\Omega} \rho_0 \frac{\abs{\dt v}^2}{2} + \frac{P'(\rho_0) \rho_0}{2} \abs{\diverge{v} - \frac{g}{P'(\rho_0)} v_3 }^2 + \int_{\Omega} \frac{\ep_0}{2} \abs{D \dt v + D \dt v^T - \frac{2}{3} (\diverge{\dt v})I  }^2 \\
+ \int_\Omega \delta_0 \abs{\diverge{\dt v}}^2
= \dt \int_{\Rn{2}} g \jump{\rho_0} \frac{\abs{v_3}^2}{2} 
- \int_{\Rn{2}} \jump{ T e_3  \cdot \dt v}.
\end{multline}
Using the jump conditions, we find that
\begin{multline}
- \int_{\Rn{2}} \jump{ T e_3  \cdot \dt v}
= \int_{\Rn{2}} \sigma \Delta_{x_1,x_2} v_3  \dt v_3\\
 = - \sigma \int_{\Rn{2}} \nab_{x_1,x_2} v_3 \cdot \nab_{x_1,x_2} \dt v_3 = - \dt \int_{\Rn{2}} \frac{\sigma}{2} \abs{\nab_{x_1,x_2} v_3 }^2.
\end{multline}
The result follows by plugging this in above.
\end{proof}


The next result allows us to estimate the energy in terms of $\Lambda$, which was given by \eqref{max_def}.

\begin{lem}\label{lin_en_bound}
Let $v\in H^1(\Omega)$ be so that $v(x_1,x_2,-m)=v(x_1,x_2,\ell)=0$.  In the non-periodic case we have the inequality
\begin{multline}
\int_{\Rn{2}} \frac{g \jump{\rho_0}}{2} \abs{v_3}^2 - \frac{\sigma}{2}\abs{\nab_{x_1,x_2} v_3}^2 - \int_\Omega \frac{P'(\rho_0)\rho_0}{2}\abs{\diverge{v} - \frac{g}{P'(\rho_0)}v_3 }^2 \\
\le \frac{\Lambda^2}{2}\int_\Omega \rho_0 \abs{v}^2
+ \frac{\Lambda}{2} \int_\Omega \frac{\ep_0}{2} \abs{D v + Dv^T - \frac{2}{3}(\diverge{v})I  }^2 + \delta_0 \abs{\diverge{ v}}^2 . 
\end{multline}
In the periodic case, if $\sqrt{\sigma / (g\jump{\rho_0})}< L$, then the same inequality holds with the $\Rn{2}$ integral replaced with an integral over $(2\pi L \mathbb{T})^2$ and $\Lambda$ replaced with $\Lambda_L$.  
 \end{lem}
\begin{proof}
We will again prove only the non-periodic version.   Take the horizontal Fourier transform and apply \eqref{parseval} to see that 
\begin{multline}
4\pi^2 \int_{\Rn{2}} \frac{g \jump{\rho_0}}{2} \abs{v_3}^2 - \frac{\sigma}{2}\abs{\nab_{x_1,x_2} v_3}^2 - 4\pi^2 \int_\Omega \frac{P'(\rho_0)\rho_0}{2}\abs{ \diverge{v} - \frac{g}{P'(\rho_0)}v_3 }^2 \\
= \int_{\Rn{2}} \frac{g \jump{\rho_0} - \sigma \abs{\xi}^2}{2} \abs{\hat{v}_3}^2 -  \int_\Omega \frac{P'(\rho_0) \rho_0}{2}\abs{ i\xi_1 \hat{v}_1 + i\xi_2 \hat{v}_2 + \partial_3 \hat{v}_3 - \frac{g}{P'(\rho_0)}\hat{v}_3 }^2 d\xi dx_3\\
=  \int_{\Rn{2}} \left( \frac{g \jump{\rho_0} - \sigma \abs{\xi}^2 }{2} \abs{\hat{v}_3}^2- \int_{-m}^\ell \frac{P'(\rho_0) \rho_0}{2}\abs{ i\xi_1 \hat{v}_1 + i\xi_2 \hat{v}_2 + \partial_3 \hat{v}_3 - \frac{g}{P'(\rho_0)}\hat{v}_3 }^2 dx_3  \right)d\xi.
\end{multline}
Consider now the last integrand for fixed $\xi\neq 0$, writing $\varphi(x_3)= i \hat{v}_1(\xi,x_3)$, $\theta(x_3)= i \hat{v}_2(\xi,x_3)$, $\psi(x_3) = \hat{v}_3(\xi,x_3)$.  That is, define
\begin{equation}
Z(\varphi,\theta,\psi;\xi) = \frac{g \jump{\rho_0} - \sigma \abs{\xi}^2}{2} \abs{\psi}^2 - \int_{-m}^\ell \frac{P'(\rho_0)\rho_0}{2}\abs{ \xi_1 \varphi + \xi_2 \theta + \psi' - \frac{g}{P'(\rho_0)} \psi }^2 dx_3
\end{equation}
where $' = \partial_3$.  By splitting
\begin{equation}
 Z(\varphi,\theta,\psi;\xi) = Z(\Re \varphi,\Re \theta,\Re \psi;\xi) +Z(\Im \varphi,\Im \theta,\Im \psi;\xi) 
\end{equation}
we may reduce to bounding $Z$ when $\varphi,\theta,\psi$ are real-valued functions, and then apply the bound to the real and imaginary parts of $\varphi,\theta,\psi$. 

The expression for $Z$ is invariant under simultaneous rotations of $\xi$ and $(\varphi,\theta)$, so without loss of generality we may assume that $\xi = (\abs{\xi},0)$ with $\abs{\xi} > 0$ and $\theta =0$.   If $\sigma>0$ then we assume for now that $\abs{\xi} < \abs{\xi}_c$ as well.  Then, using \eqref{E_def} with $\tep = \lambda(\abs{\xi})\ep_0$ and $\td = \lambda(\abs{\xi}) \delta_0$, we may rewrite 
\begin{multline}
 Z(\varphi,\theta,\psi;\xi) = -E(\varphi,\psi;\lambda(\abs{\xi})).+ \frac{\lambda(\abs{\xi})}{2} \int_{-m}^\ell \delta_0 \abs{\psi'+\abs{\xi} \varphi}^2 \\
+ \frac{\lambda(\abs{\xi})}{2} \int_{-m}^\ell \ep_0 \left(  \abs{\varphi' - \abs{\xi} \psi}^2 + \abs{\psi' -\abs{\xi} \varphi}^2 + \frac{1}{3} \abs{\psi' + \abs{\xi} \varphi}^2\right)
\end{multline}
and hence
\begin{multline}
 Z(\varphi,\theta,\psi;\xi) \le \frac{\Lambda^2}{2} \int_{-m}^\ell \rho_0 (\abs{\varphi}^2  +\abs{\psi}^2) \\
+ \frac{\Lambda}{2} \int_{-m}^\ell \delta \abs{\psi'+\abs{\xi} \varphi}^2 
+ \frac{\Lambda}{2} \int_{-m}^\ell \ep_0 \left(  \abs{\varphi' - \abs{\xi} \psi}^2 + \abs{\psi' -\abs{\xi} \varphi}^2 + \frac{1}{3} \abs{\psi' + \abs{\xi} \varphi}^2\right)
\end{multline}
For $\abs{\xi} \ge \xi_c$ the expression for $Z$ is non-positive, so the previous inequality holds trivially, and so we deduce that it holds for all $\abs{\xi}>0$.

Translating the inequality back to the original notation for fixed $\xi$, we find
\begin{multline}
 \frac{g \jump{\rho_0} - \sigma \abs{\xi}^2}{2} \abs{\hat{v}_3}^2 - \int_{-m}^\ell \frac{P'(\rho_0) \rho_0}{2}\abs{ i\xi_1 \hat{v}_1 + i\xi_2 \hat{v}_2 + \partial_3 \hat{v}_3 - \frac{g}{P'(\rho_0)}\hat{v}_3 }^2 dx_3 \\
\le \frac{\Lambda^2}{2} \int_{-m}^\ell \rho_0 \abs{\hat{v}}^2 
+ \frac{\Lambda}{2} \int_{-m}^\ell \delta_0 \abs{i\xi_1 \hat{v}_1 + i\xi_2 \hat{v}_2 + \partial_3 \hat{v}_3 }^2 + \frac{\ep_0}{2} \abs{\hat{B}}^2,
\end{multline}
where 
\begin{equation}
 B = Dv + Dv^T - \frac{2}{3}(\diverge{v})I.
\end{equation}
Integrating each side of this inequality over all $\xi \in\Rn{2}$ and using \eqref{parseval} then proves the result.
\end{proof}

When $\sigma>0$ and  $L$ is sufficiently small, a better result is available in the periodic case. 

\begin{lem}\label{periodic_bound}
Let $v\in H^1(\Omega)$ be so that $v(x_1,x_2,-m)=v(x_1,x_2,\ell)=0$ and suppose in the periodic case that $L$ satisfies \eqref{L_small}.  Then
\begin{equation}\label{p_b_1}
\int_{(2\pi L\mathbb{T})^2} \frac{g \jump{\rho_0}}{2} \abs{v_3}^2 - \frac{\sigma}{2}\abs{\nab_{x_1,x_2} v_3}^2 - \int_\Omega \frac{P'(\rho_0)\rho_0}{2}\abs{\diverge{v} - \frac{g}{P'(\rho_0)}v_3 }^2 
\le 0.
\end{equation}

\end{lem}
\begin{proof}
Apply the horizontal Fourier transform to see
\begin{multline}
4\pi^2L^2 \int_{(2\pi L\mathbb{T})^2} \frac{g \jump{\rho_0}}{2} \abs{v_3}^2 - \frac{\sigma}{2}\abs{\nab_{x_1,x_2} v_3}^2 - 4\pi^2L^2 \int_\Omega \frac{P'(\rho_0)\rho_0}{2}\abs{ \diverge{v} - \frac{g}{P'(\rho_0)}v_3 }^2 \\
=  \sum_{\xi \in (L^{-1}\mathbb{Z})^2}  \frac{g \jump{\rho_0} - \sigma \abs{\xi}^2 }{2} \abs{\hat{v}_3}^2 \\
- \sum_{\xi \in (L^{-1}\mathbb{Z})^2}   \int_{-m}^\ell \frac{P'(\rho_0) \rho_0}{2}\abs{ i\xi_1 \hat{v}_1 + i\xi_2 \hat{v}_2 + \partial_3 \hat{v}_3 - \frac{g}{P'(\rho_0)}\hat{v}_3 }^2 dx_3  .
\end{multline}
Because of \eqref{L_small}, the only $\xi \in (L^{-1}\mathbb{Z})^2$ for which $g \jump{\rho_0} - g \abs{\xi}^2 \ge 0$ is $\xi =0$.  Since all but the $\xi=0$ term on the right side of the last equation are non-positive, we reduce to showing that 
\begin{equation}
  \frac{g \jump{\rho_0} }{2} \abs{\hat{v}_3}^2 
- \int_{-m}^\ell \frac{P'(\rho_0) \rho_0}{2}\abs{ \partial_3 \hat{v}_3 - \frac{g}{P'(\rho_0)}\hat{v}_3 }^2 dx_3 \le 0.
\end{equation}
For this we expand the term in the integral and integrate by parts to get 
\begin{equation}
 \frac{g \jump{\rho_0} }{2} \abs{\hat{v}_3}^2 
- \int_{-m}^\ell \frac{P'(\rho_0) \rho_0}{2}\abs{ \partial_3 \hat{v}_3 - \frac{g}{P'(\rho_0)}\hat{v}_3 }^2 dx_3  = -\hal \int_{-m}^\ell P'(\rho_0) \rho_0 \abs{\partial_3 \hat{v}_3}^2, 
\end{equation}
which yields the desired inequality.

\end{proof}

\subsection{Proof of Theorems \ref{lin_growth_bound} and \ref{periodic_stability}}\label{linear_growth_section}

With the preliminary estimates in place, we can now prove bounds for the growth of arbitrary solutions to \eqref{second_order} in terms of $\Lambda$ and $\Lambda_L$.

\begin{proof}[Proof of Theorem \ref{lin_growth_bound}]
Again, we will only prove the non-periodic case.  Integrate the result of Lemma \ref{lin_en_evolve} in time from $0$ to $t$ to find that
\begin{multline}
  \int_\Omega \rho_0 \frac{\abs{\dt v(t)}^2}{2} + \int_{0}^{t}\int_\Omega \frac{\ep_0}{2} \abs{D \dt v(s) + D \dt v(s)^T - \frac{2}{3} (\diverge{\dt v(s)})I}^2 + \delta_0 \abs{\diverge{\dt v(s)}}^2 ds \\
\le K_0 + \int_{\Rn{2}} \frac{g \jump{\rho_0}}{2} \abs{v_3(t)}^2 - \frac{\sigma}{2}\abs{\nab_{x_1,x_2} v_3(t)}^2 - \int_\Omega \frac{P'(\rho_0) \rho_0}{2}\abs{ \diverge{v(t)} - \frac{g}{P'(\rho_0)}v_3(t) }^2,
\end{multline}
where 
\begin{equation}
 K_0 = \int_\Omega \rho_0 \frac{\abs{\dt v(0)}^2}{2} + \int_\Omega \frac{P'(\rho_0) \rho_0}{2}\abs{ \diverge{v(0)} - \frac{g}{P'(\rho_0)} v_3(0) }^2 + \int_{\Rn{2}} \frac{\sigma}{2}\abs{\nab_{x_1,x_2} v_3(0)}^2.
\end{equation}
We may then apply Lemma \ref{lin_en_bound} to get the inequality 
\begin{multline}
  \int_\Omega \rho_0 \frac{\abs{\dt v(t)}^2}{2} + \int_{0}^{t}\int_\Omega \frac{\ep_0}{2} \left( D \dt v(s) + D \dt v(s)^T - \frac{2}{3} (\diverge{\dt v(s)})I \right) + \delta_0 \abs{\diverge{\dt v(s)}}^2 ds\\
\le K_0 + \frac{\Lambda^2}{2}\int_\Omega \rho_0 \abs{v(t)}^2 \\
+ \frac{\Lambda}{2}\int_\Omega \frac{\ep_0}{2} \left( D \dt v(t) + D \dt v(t)^T - \frac{2}{3} (\diverge{\dt v(t)})I \right) + \delta_0 \abs{\diverge{v(t)}}^2.
\end{multline}
Using the definitions of the norms $\norm{\cdot}_1, \norm{\cdot}_2$ given in \eqref{norm_def}, we may compactly rewrite the previous inequality as
\begin{equation}\label{l_g_b_1}
 \hal \norm{\dt v(t)}^2_1  + \int_0^t \norm{\dt v(s)}^2_2 ds \le K_0 + \frac{\Lambda^2}{2}\norm{ v(t)}^2_1 + \frac{\Lambda}{2} \norm{v(t)}^2_2. 
\end{equation}

Integrating in time and using Cauchy's inequality,  we may bound
\begin{multline}\label{l_g_b_2}
 \Lambda \norm{v(t)}^2_2 = \Lambda \norm{v(0)}^2_2 +\Lambda \int_0^t 2 \langle v(s),\dt v(s) \rangle_2 ds\\
 \le \Lambda \norm{v(0)}^2_2 +  \int_0^t \norm{\dt v(s)}_2^2  ds  + \Lambda^2 \int_0^t \norm{v(s)}^2_2 ds .
\end{multline}
On the other hand
\begin{equation}
\Lambda \dt \norm{v(t)}^2_1 = \Lambda 2\langle \dt v(t), v(t) \rangle_1 \le \Lambda^2 \norm{v(t)}_1^2 + \norm{\dt v(t)}^2_1.
\end{equation}
We may combine these two inequalities with \eqref{l_g_b_1} to derive the differential inequality
\begin{equation}
 \dt   \norm{v(t)}^2_1 +  \norm{v(t)}^2_2 \le K_1 + 2\Lambda \norm{v(t)}^2_1 + 2\Lambda \int_0^t \norm{v(s)}^2_2 ds
\end{equation}
for $K_1 = 2K_0/\Lambda + 2\norm{v(0)}^2_2.$  An application of Gronwall then shows that
\begin{equation}\label{l_g_b_3}
 \norm{v(t)}^2_1 + \int_0^t \norm{v(s)}^2_2 ds \le e^{2\Lambda t}\norm{v(0)}^2_1 + \frac{K_1}{2\Lambda} (e^{2\Lambda t}-1)
\end{equation}
for all $t\ge 0$.  To derive the corresponding bound for $\norm{v(t)}^2_2$ and $\norm{\dt v(t)}_1^2$ we return to \eqref{l_g_b_1} and plug in \eqref{l_g_b_2} and \eqref{l_g_b_3} to see that
\begin{equation}
 \frac{1}{\Lambda} \norm{\dt v(t)}_1^2 +   \norm{v(t)}^2_2 \le K_1 + \Lambda \norm{v(t)}_1^2 +  2\Lambda \int_0^t \norm{v(s)}^2_2 ds \le  e^{2\Lambda t} \left( 2\Lambda \norm{v(0)}^2_1 + K_1 \right).
\end{equation}
The result follows by noting that 
\begin{equation}
 K_0 \le C \left( \norm{\dt v(0)}_1^2 + \norm{ v(0)}_1^2 + \norm{v(0)}_2^2  + \sigma  \int_{\Rn{2}} \abs{\nab_{x_1,x_2} v_3(0)}^2 \right) 
\end{equation}
for a constant $C>0$ depending on  $\rho_0^{\pm}, P_{\pm}, \Lambda, \ep_\pm, \delta_\pm, \sigma, g, m, \ell.$

\end{proof}

In the periodic case when $L$ satisfies \eqref{L_small} we may use Lemma \ref{periodic_bound} to improve the above result.

\begin{proof}[Proof of Theorem \ref{periodic_stability}]
We again integrate the result of Lemma \ref{lin_en_evolve} in time from $0$ to $t$ to find that
\begin{multline}
  \int_\Omega \rho_0 \frac{\abs{\dt v(t)}^2}{2} + \int_{0}^{t}\int_\Omega \frac{\ep_0}{2} \abs{D \dt v(s) + D \dt v(s)^T - \frac{2}{3} (\diverge{ \dt v(s)})I}^2 + \delta_0 \abs{\diverge{ \dt v(s)}}^2 ds \\
\le K_1 + \int_{(2\pi L \mathbb{T})^2} \frac{g \jump{\rho_0}}{2} \abs{v_3(t)}^2 - \frac{\sigma}{2}\abs{\nab_{x_1,x_2} v_3(t)}^2 - \int_\Omega \frac{P'(\rho_0) \rho_0}{2}\abs{\diverge{v(t)} - \frac{g}{P'(\rho_0)} v_3(t) }^2.
\end{multline}
We may apply Lemma \ref{periodic_bound} to see that all of the integrals on the right side of the previous inequality are non-positive, and hence
\begin{equation}
\hal \norm{\dt v(t) }_1^2 + \int_0^t \norm{\dt v(s)}_2^2 ds \le K_1,
\end{equation}
where the norms are defined by \eqref{norm_def}.  From this we deduce that
\begin{equation}
 \norm{v(t)}_1  + \norm{v(t)}_2 \le \norm{v(0)}_1 + \norm{v(0)}_2 + 3\sqrt{t}\sqrt{K_1}.
\end{equation}
Then, using that $\dt \eta = v$, we get
\begin{equation}
 \norm{\eta(t)}_1  + \norm{\eta(t)}_2 \le \norm{\eta(0)}_1 + \norm{\eta(0)}_2 + t\left(\norm{v(0)}_1 + \norm{v(0)}_2 \right) + 2t^{3/2} \sqrt{K_1}.
\end{equation}

To derive the estimates for $\dt^j v$ for $j\ge 2$ we apply $\dt^j$ to \eqref{second_order}.  Then $w = \dt^j v$ satisfies the same equation and boundary conditions as $v$, which allows us to argue as above to derive the inequality
\begin{equation}
\hal \norm{\dt^j v(t) }_1^2 + \int_0^t \norm{\dt^j v(s)}_2^2 ds \le K_j
\end{equation}
for  all $j\ge 1$.  This trivially implies \eqref{p_s_0}.  To get \eqref{p_s_00} we bound
\begin{multline}
 \norm{\dt^j v(t)}^2_2 \le  \norm{\dt^j v(0)}^2_2 + 2\int_0^t \norm{\dt^j v(s)}_2 \norm{\dt^{j+1}  v(s)}_2 ds \\
\le  \norm{\dt^j v(0)}^2_2 + 2\left(\int_0^t \norm{\dt^j v(s)}^2_2 ds\right)^{1/2}\left( \int_0^t \norm{\dt^{j+1}  v(s)}^2_2 ds\right)^{1/2} \\
\le \norm{\dt^j v(0)}^2_2 + 2 \sqrt{K_j}\sqrt{K_{j+1}}.
\end{multline}
\end{proof}


\end{document}